\documentclass[12 pt, leqno]{amsart}

\usepackage{amsfonts}
\usepackage{amsthm}
\usepackage{amssymb}
\usepackage{latexsym}
\usepackage{amsmath}
\usepackage{mathrsfs}
\usepackage[all]{xy}
\usepackage{comment}

\theoremstyle{plain}
\newtheorem{tw}{Theorem}[section]

\newtheorem {lem} [tw]{Lemma}
\newtheorem {prop}[tw] {Proposition}

\newtheorem{cor}[tw]{Corollary}

\theoremstyle{definition}
\newtheorem {deft}[tw] {Definition}
\newtheorem {rem} [tw]{Remark}
\newtheorem {example}[tw] {Example}

\newcommand{\bc} {\Bbb C}
\newcommand{\bn}{\Bbb N}

\newcommand{\bF}{\Bbb F}
\newcommand{\bt}{\Bbb T}
\newcommand{\bz}{\Bbb Z}

\newcommand{\OA}{\mathsf{O}_A}
\newcommand{\On}{\mathsf{O}_n}

\newcommand{\alg} {\mathsf{A}}
\newcommand{\blg} {\mathsf{B}}
\newcommand{\dlg} {\mathsf{D}}

\newcommand{\Alg} {\mathcal{A}}
\newcommand{\Dlg} {\mathcal{D}}
\newcommand{\Blg} {\mathcal{B}}

\newcommand {\QG} {{\mathbb{G}}}
\newcommand {\QH} {{\mathbb{H}}}

\newcommand {\id} {{\textup{id}}}
\newcommand {\QISO} {{\mathsf{QISO}}}
\newcommand {\QISOb} {{\mathsf{QISO}}_b}
\newcommand {\QIb} {{\mathsf{QISO}}_{(b)}}

\newcommand{\tu}{\textup}

\newcommand{\catV}{\mathcal{C}_{\alg,\mathcal{V}}}

\newcommand{\Hil}{\mathsf{H}}

\newcommand{\Ind}{\mathcal{I}}
\newcommand{\Jnd}{\mathcal{J}}

\newcommand{\cat}{\mathcal{C}}
\newcommand{\Seq}{\mathcal{S}}

\newcommand{\Com}{\Delta}

\newcommand{\freedual}{\widehat{\bF_n}}
\newcommand{\Irr}{\textup{Irr}}

\newenvironment{rlist}
{

\begin{enumerate}}
{\end{enumerate}}

\newcommand{\ot}{\otimes}
\newcommand{\ida}{1_{\alg}}
\newcommand{\wt}{\widetilde}

\numberwithin{equation}{section}

\begin{document}

\author{Teodor Banica}
\address{T.B.: Department of Mathematics, Cergy-Pontoise University, 95000 Cergy-Pontoise, France. {\tt teodor.banica@u-cergy.fr}}

\author{Adam Skalski}
\address{A.S.: Institute of Mathematics of the Polish Academy of Sciences, ul. \'Sniadeckich 8, 00-956 Warszawa, Poland. {\tt a.skalski@impan.pl}}

\keywords{Quantum symmetry group, filtered algebra, partial isometry, non-crossing partitions}
\subjclass[2000]{Primary 46L65, Secondary 16W30, 46L54}

\title[Quantum symmetry groups]{\bf Quantum symmetry groups of $C^*$-algebras equipped with orthogonal filtrations}

\begin{abstract}
Motivated by the work of Goswami on quantum isometry groups of noncommutative manifolds we define the quantum symmetry group of a unital $C^*$-algebra $\alg$ equipped with an orthogonal filtration as the universal object in the category of compact quantum groups acting on $\alg$ in a filtration preserving fashion. The existence of such a universal object is proved and several examples discussed. In particular we study the universal quantum group acting on the dual of the free group and preserving both the word length and the block length. 
\end{abstract}

\maketitle

\section{Introduction}

Groups first appeared in mathematics, in the second half of the 19th century, as collections of symmetries of some structure: a finite set, a figure on the plane, a set of solutions of a given equation, etc..\  When compact quantum groups entered the scene over a hundred years later, after many earlier developments due to the work of Kac, Vainerman, Enock, Schwartz, Takesaki, Drinfeld and others, the initial examples were rather constructed via  algebraic methods (\cite{wo1}). This was very natural, as the approach to quantum groups based on the Gelfand-Najmark duality means that in fact instead of studying directly a compact quantum group $\QG$ we rather work with the algebra $C(\QG)$ playing the role of the algebra of continuous functions on $\QG$. It is therefore to be expected that examples can be obtained by considering algebraic deformations of the algebra of functions on a classical group. On the other hand, from the modern perspective, a symmetry group of a given structure $X$ can be viewed as the universal, final object in the category of groups acting on $X$. This observation, together with the fact that soon after the fundamental paper of Woronowicz the notion of an action of a compact quantum group on a $C^*$-algebra was introduced in \cite{Podles} and \cite{Boca}, opened the road to defining quantum symmetry groups.

The first developments in this area are due to Wang, who in \cite{wa2} defined the quantum symmetry group of a finite set (the quantum permutation group $S^+_n$) and the quantum symmetry group of a finite-dimensional $C^*$-algebra equipped with a faithful state. Already here many interesting aspects of the theory manifested -- for example it turned out that the quantum permutation group $S^+_n$ is infinite, as soon as $n\geq 4$. The study of quantum symmetry groups of finite structures initiated by Wang was later extended by the first named author, Bichon, Collins and others. In particular quantum symmetry groups of finite graphs and finite metric spaces (see for example \cite{banmet}, \cite{BanBic}) were introduced in their work. This led to several new examples of quantum groups, provided methods of computing their representation theory and exposed  multiple connections to free probability. These investigations were moreover one of the natural impulses for the study of \emph{liberated quantum groups} (\cite{bsp}), which are free counterparts of some classical compact groups.

Up to 2009 the quantum symmetry groups were only defined for various types of finite structures. The next important step in the development of the theory arrived  when Goswami proposed in \cite{gos} a definition of the quantum isometry group of a noncommutative compact  manifold \`a la Connes, thus generalizing the classical notion of the isometry group of a compact Riemannian manifold. In this ground-breaking paper, and in the following works (\cite{bg1}, \cite{JyotDeb2}, \cite{bgs}, \cite{bhs}) many examples of such quantum isometry groups were computed. Interestingly, it turned out that in all examples studied so far quantum isometry groups of classical connected manifolds coincide with classical ones. In \cite{bhs} a particular focus was put on the class of noncommutative manifolds related to duals of finitely generated discrete groups More specifically, Bhowmick and the second named author introduced there the notion of the quantum symmetry group acting on the  dual of a finitely generated discrete group in a \emph{word length} preserving fashion. We continued the study of such quantum groups, concentrating on the free products of cyclic groups, in \cite{bsk1} and \cite{bsk2}, where we discovered many connections with the theory of liberated quantum groups and with free probabilistic concepts. Another recent example of the computations of quantum symmetry groups of that type can be found in \cite{LDS}.

Motivated by the ideas of Goswami and his collaborators we develop in this paper a \emph{general theory of quantum symmetry groups of $C^*$-algebras equipped with orthogonal filtrations}. The filtrations with which we work consist of finite-dimensional subspaces. This can be viewed as a natural `compactness' property of the possibly infinite structure we consider and allows us to establish that the universal object in the corresponding category of quantum group actions preserving a given filtration exist (note that although for simplicity we consider here only the category of actions of compact quantum groups, we could also allow actions of arbitrary locally compact quantum groups, following the discussion after Theorem 1.1 in \cite{LDS}). The main theorem is the following (the detailed statement and all relevant definitions can be found in Section 2).

\smallskip
\noindent\textbf{Theorem.} \emph{ Let $\alg$ be a unital $C^*$-algebra with a faithful state $\omega$ and let $\mathcal{V}$ be a filtration of $\alg$, i.e.\ a collection of finite-dimensional subspaces of $\alg$, which are pairwise mutually orthogonal with respect to the scalar product given by $\omega$, and together span a dense $^*$-subalgebra of $\alg$. Then there exists a universal compact quantum group acting on $\alg$ in such a way that the action preserves the filtration $\mathcal{V}$.}

\smallskip
The resulting framework turns out to be very general. In particular it covers
\begin{itemize}
\item all the quantum symmetry groups of finite structures (sets, graphs, finite-dimensional $C^*$-algebras) studied before
\item quantum symmetry groups of AF algebras associated with their Bratteli diagrams
\item quantum symmetry groups of graph $C^*$-algebras
\item quantum symmetry groups of filtrations induced by \emph{ergodic} actions of compact (quantum) groups
\item  several classes of quantum symmetry groups associated to duals of discrete groups; in particular it allows us to extend the study of the length preserving actions initiated in \cite{bhs} to, for example, the actions which preserve the so-called \emph{block length}
\item even more importantly, it also opens the way to replacing the duals of discrete groups by arbitrary compact quantum groups, using the notion of the length on discrete quantum groups developed in \cite{Vergnioux}.
\end{itemize}

As a particular, very concrete example we show that the quantum symmetry group acting on the dual of the free group $\mathbb{F}_n$ in the block length preserving fashion is $K_n^+$, a compact quantum group first discovered in \cite{bsk2} from the representation theory perspective. Together with the quantum groups $H^+(p,q)$ and  $S^+(p,q)$  first studied in \cite{bsk1}, $K^+_n$ is very closely connected to the class of \emph{easy} quantum groups investigated in \cite{bsp} and \cite{bcs}.

The detailed plan of the paper is as follows: we begin with a short preliminaries subsection, in which we recall the basic definitions related to compact quantum groups and agree the notation. The second section contains the main result of the paper, Theorem \ref{qsymV}, establishing the existence of the universal compact quantum group acting on a $C^*$-algebra equipped with an orthogonal filtration. In Section 3 we describe several examples in which the main result applies and compare the framework of orthogonal filtrations with that considered in the theory of quantum isometry groups introduced in \cite{JyotDeb2}. Section 4 continues the discussion of examples, focusing on the case where the algebra in question is the group $C^*$-algebra of a discrete group $\Gamma$ and the filtration comes from a partition of $\Gamma$; here we show for example that the quantum isometry group of the dual of $\bz_2^{\oplus k}$ is the quantum group $O_k^{-1}$ studied in \cite{bbc}. In Section 5 we specialise even further, show that the quantum symmetry group acting on the dual of the free group $\mathbb{F}_n$ in the block length preserving fashion is $K_n^+$, and describe its representation theory. Finally we mention that $K^+_n$ and other compact quantum groups discovered in \cite{bsk1} can be viewed as \emph{super-easy} quantum groups,  a class naturally extending that of easy quantum groups studied in \cite{bsp} and \cite{bcs}.

\subsection{Preliminaries}

The algebraic tensor product will be denoted by $\odot$, the symbol $\ot$ will be reserved for tensor products of maps and for spatial tensor products of $C^*$-algebras.

The following definitions date back to the fundamental papers \cite{wo1} and \cite{wo2}.

\begin{deft}
A unital $C^*$-algebra $\alg$ equipped with a unital $^*$-homomorphism $\Com:\alg \to \alg \ot \alg$ which is coassociative:
\[ (\Com \ot \ida) \Com = (\ida \ot \Com) \Com\]
and satisfies the quantum cancellation rules:
\[ \overline{\tu{Lin}}\, \Com(\alg) (1 \ot \alg) = \overline{\tu{Lin}}\, \Com(\alg) (\alg \ot 1) = \alg \ot \alg\]
is called the algebra of continuous functions on a compact quantum group. We usually write $\alg=C(\QG)$ and informally call $\QG$ a compact quantum group.
\end{deft}

\begin{deft}
A unitary matrix $U=(U_{ij})_{i,j=1}^n \in M_n (C(\QG))$ is called a unitary representation of $\QG$ if for each $i,j=1,\ldots,n$
\[ \Com(U_{ij})=\sum_{k=1}^n U_{ik} \ot U_{kj}.\]
If $\{U_{ij}:i,j=1,\ldots,n\}$, the set of \emph{coefficients} of $U$, generates $C(\QG)$ as a $C^*$-algebra, we call it a fundamental representation (and say that $\QG$ is a compact matrix quantum group). A unitary representation is called irreducible if the only matrices in $M_n$ commuting with $U$ are multiples of the identity matrix.
\end{deft}

The complete set of unitary equivalence classes of irreducible unitary representations of $\QG$ will be denoted by  $\textup{Irr} (\QG)$.
The span of all coefficients of unitary representations of $\QG$ forms a dense unital $^*$-subalgebra of $C(\QG)$, denoted $R(\QG)$ (and in fact equipped with a canonical Hopf $^*$-algebra structure). Each compact quantum group admits a unique \emph{Haar state}, i.e.\ a state $h\in C(\QG)^*$ such that
\[ (h \ot \ida)\Com = (\ida \ot h)\Com = h(\cdot) 1_{C(\QG)}.\]
The Haar state need not be faithful (unless $\QG$ is \emph{coamenable}). We will at some point need to consider the \emph{reduced} version of $C(\QG)$, denoted by $C_r(\QG)$ and arising as the image of $C(\QG)$ under the GNS representation with respect to the Haar state, and the \emph{universal} version of $C(\QG)$, denoted by $C_u(\QG)$, and arising as the enveloping $C^*$-algebra of $R(\QG)$. For more information on this we refer to \cite{coamen}.

\section{Quantum isometry groups associated to $C^*$-algebras equipped with an orthogonal filtration}

In our previous work (\cite{bhs}, \cite{bsk1}, \cite{bsk2}) we studied quantum isometry groups related to the word-length preserving actions on duals of discrete groups. The existence of such isometry groups is a consequence of the general results of \cite{JyotDeb2}, where Bhowmick and Goswami defined quantum groups of orientation preserving isometries.

In this section we show how we can in general associate a quantum symmetry group with each $C^*$-algebra equipped with an `orthogonal' filtration. The special cases of this construction have been considered in \cite{bgs} (quantum isometry groups of AF algebras) and in \cite{bhs} (afore-mentioned universal quantum groups acting on duals of finitely generated discrete groups in a word-length preserving way). The proof of the existence is inspired by the results in Section 4 in \cite{gos}, but we also use some ideas of \cite{Sabbe}. Some comments on the relation of this work to previous results available in literature will be made after the theorem is proved and  in the following section.

\begin{deft} \label{orthfilt}
Let $\alg$ be a unital $C^*$-algebra equipped with a faithful state $\omega$ and with a family $(V_i)_{i\in \Ind}$ of finite-dimensional subspaces of $\alg$ (with the index set $\Ind$ containing a distinguished element 0) satisfying the following conditions:
\begin{rlist}
\item $V_0=\bc 1_{\alg}$;
\item for all $i,j\in \Ind$, $i\neq j$, $a\in V_i$ and $b\in V_j$ we have $\omega(a^*b)=0$;
\item the set $\textup{Lin} (\bigcup_{i\in \Ind} V_i)$ 
is a dense $^*$-subalgebra of $\alg$.
\end{rlist}
If the above conditions are satisfied we say that the pair $(\omega, (V_i)_{i\in \Ind})$ defines an orthogonal filtration of $\alg$; sometimes abusing the notation we will omit $\omega$ and simply say that $(\alg, (V_i)_{i\in \Ind})$ is a $C^*$-algebra with an orthogonal filtration. The (dense) $^*$-subalgebra spanned in $\alg$ by $\{V_i:i\in \Ind\}$ will be denoted by $\Alg$.
\end{deft}

Note that the existence of an orthogonal filtration does not imply that the $C^*$-algebra $\alg$ is AF (although of course unital separable AF $C^*$-algebras admit orthogonal filtrations). Other examples of importance for us are the reduced group $C^*$-algebras; this will be explained in detail in Section \ref{group}. In most cases we have in fact $V_i=V_i^*$ and $\omega$ is a trace. Note that $\alg$ can be viewed as the completion of $\Alg$ in the GNS representation with respect to $\omega$. 



The following definition comes from \cite{Podles}.

\begin{deft} \label{defcat}
Let $\alg$ be a unital $C^*$-algebra. We say that a compact quantum group $\QG$ acts on $\alg$ if there exists a unital $^*$-homomorphism $\alpha:\alg \to \alg \ot C(\QG)$, called the action of $\QG$ on $\alg$,  such that
\begin{rlist}
\item $(\alpha \ot \id_{C(\QG)}) \alpha = (\id_{\alg} \ot \Com) \alpha;$
\item $\overline{\tu{Lin}}\, \alpha(\alg) (1 \ot C(\QG)) = \alg \ot C(\QG).$
\end{rlist}
\end{deft}

\begin{deft}
Let $(\alg, \omega, (V_i)_{i\in \Ind})$ be a $C^*$-algebra with an orthogonal filtration. We say that a quantum group $\QG$ acts on $\alg$ in a filtration preserving way if there exists an action $\alpha$ of $\QG$ on $\alg$ such that the following condition holds:
\[\alpha(V_i) \subset V_i \odot C(\QG), \;\; i\in \Ind.\]
We will then write $(\alpha,\QG)\in \catV$.
\end{deft}

Note that $\alpha$ as above automatically preserves the sets $V_i^*$ ($i\in \Ind$). It is also easy to see that if $(\alpha, \QG)\in \catV$, then $\alpha$ preserves the state $\omega$:
\begin{equation}\label{statepreserve} (\omega \ot \id_{C(\QG)}) \circ \alpha = \omega(\cdot)1_{C(\QG)}.\end{equation}
Indeed, the conditions (i) and (iii) in Definition \ref{orthfilt} imply that $\omega(a)=0$ for all $a\in \tu{Lin}\,(\cup_{i\in \Ind\setminus\{0\}} V_i)$. Hence the equality \eqref{statepreserve} holds on the dense subalgebra $\Alg$; as both sides of \eqref{statepreserve} are continuous, it must in fact hold everywhere.

Before we continue, we make one important observation.  Let $(\alpha,\QG) \in \catV$. It is not difficult to check that in fact for each $i\in \Ind$ we have
\[\alpha(V_i) \subset V_i \odot R(\QG)\]
(in fact in the first two steps of the proof of Theorem \ref{qsymV} below we show that in a certain natural sense $\alpha|_{V_i}$ induces a $\textup{dim}(V_i)$-dimensional unitary representation of $\QG$). Moreover we can always assume, by passing to the GNS representation $\pi_h$ of $C(\QG)$ with respect to the Haar state of $\QG$, that actually we are in the situation $\alpha:\alg \to \alg \ot C_r(\QG)$,  where $C_r(\QG)$ denotes the `reduced version' of $C(\QG)$. Formally speaking we define the new action by $\wt{\alpha}:=(\ida \ot \pi_h)\alpha$ (for more information on this procedure we refer to \cite{Soltan}). As this passage does not change essentially the character of the action, we will usually assume that  the quantum group actions are reduced.

The morphisms in the category $\catV$ are compact quantum group morphisms which intertwine the respective actions. Here we need to be careful, as natural compact quantum group morphisms act between \emph{universal} versions of the algebras $C(\QG)$.  This means that if $(\QG_1,\alpha_1), (\QG_2, \alpha_2) \in \catV$ then a morphism from $(\QG_1,\alpha_1)$ to $(\QG_2, \alpha_2)$ is a unital $^*$-homomorphism $\pi:C_u(\QG_2) \to C_u(\QG_1)$ such that
\[ (\pi \ot \pi) \circ \Com_2 = \Com_1 \circ \pi\]
(so that $\pi$ automatically restricts to a unital $^*$-homomorphism $\pi_0:R(\QG_2) \to R(\QG_1)$) and
\[ (\id_{\alg} \ot \pi_0)\alpha_2|_{\Alg} =\alpha_1|_{\Alg}\]
(recall that $\Alg$ is the dense subalgebra of $\alg$ spanned by elements in all subspaces $V_i$, $i\in \Ind$). Note the inversion of arrows. The subtleties related to distinguishing between the universal and reduced versions of course disappear if the quantum groups in question happen to be coamenable.

\begin{deft}
We say that $(\alpha_u, \QG_u)$ is a universal final object in $\catV$ if for any $(\alpha,\QG) \in \catV$ there exists a unique morphism $\pi$ from  $(\alpha,\QG)$ to $(\alpha_u, \QG_u)$.
\end{deft}

Note that if the  final object in $\catV$ exists, it is automatically unique up to isomorphism. Before we begin the proof of the main theorem of this section, let us observe that the category introduced above can be studied in a purely Hopf $^*$-algebraic language.

Define indeed for each $i\in \Ind$ the subspace $\Blg_i$ of $R(\QG)$ as $\{(f\ot \id)\alpha(v): v\in V_i, f \in V_i'\}$. The algebra generated by all $\Blg_i$ inside $R(\QG)$ is a Hopf $^*$-algebra, which we will denote $R_{\alpha}(\QG)$. If $R_{\alpha}(\QG)$ is dense in $C(\QG)$ (equivalently, $R_{\alpha}(\QG)$ is equal to $R(\QG)$, see \cite{DiK}), we say that the action $\alpha$ of $\QG$ on $\alg$ is \emph{faithful}.

Consider now the algebraic version of the category $\catV$,  denoted $\catV^{alg}$.

\begin{deft} \label{defcatalg}
We say that a compact quantum group $\QG$ admits an algebraic action $\alpha_0$ on $\Alg$ (the dense subalgebra of $\alg$), preserving the filtration $\mathcal{V}$ if $\alpha_0:\Alg \to \Alg \odot R(\QG)$ is a unital $^*$-homomorphism such that
\begin{rlist}
\item $ \alpha_0(\Alg) (1 \odot R(\QG)) = \Alg \odot R(\QG);$
\item $(\alpha_0 \ot \id_{R(\QG)}) \alpha_0 = (\id_{\Alg} \ot \Com) \alpha_0$;
\item $\alpha_0(V_i) \subset V_i \odot R(\QG), \;\; i \in \Ind.$
\end{rlist}
We then write $(\alpha_0,\QG) \in \catV^{alg}$. The morphisms in $\catV^{alg}$ are defined analogously to those in $\catV$, with all maps acting between the algebraic objects (so that $\pi:R(\QG_2)\to R(\QG_1)$, etc.).
\end{deft}

\begin{lem} \label{isom}
The categories $\catV$ and $\catV^{alg}$ are isomorphic.
\end{lem}

\begin{proof}
The discussion before the definition implies that if $(\alpha,\QG) \in \catV$, then $(\alpha_0:=\alpha|_{\Alg},\QG) \in \catV^{alg}$. On the other hand Lemma 3.1 of \cite{Curran} implies that if $(\alpha_0,\QG) \in \catV^{alg}$, then, as the action $\alpha_0$ preserves the (faithful) state $\omega|_{\Alg}$, and the corresponding GNS completion of $\Alg$ is isomorphic to $\alg$, $\alpha_0$ extends to an action $\alpha:\alg \to \alg \ot C_r(\QG)$ (recall that $C_r(\QG)$ denotes the `reduced version' of $C(\QG)$). It is then easy to check that $(\alpha,\QG)\in \catV$.

Given a morphism $\pi$ in $\catV$ between  $(\alpha_1,\QG_1)$ and $(\alpha_2,\QG_2)$, we know that as it is a compact quantum group morphism, it restricts to a $^*$-homomorphism between respective dense Hopf $^*$-algebras. On the other hand an `algebraic' compact quantum group morphism acting on the level of Hopf $^*$-algebras extends uniquely to a unital $^*$-homomorphism acting between their universal completions and preserving the respective coproducts. The facts that respective restrictions/extensions intertwine the respective actions follow directly from the definitions.
\end{proof}

We are ready to formulate the main theorem. Note that the first part of the proof follows the ideas of \cite{gos} (see also for example \cite{gos2}); as the proof we present is completely algebraic and hence hopefully more accessible, we decided to provide full details.

\begin{tw} \label{qsymV}
Let $(\alg, \omega, (V_i)_{i\in \Ind})$ be a $C^*$-algebra with an orthogonal filtration. There exists a final object in the category $\catV$; in other words there exists a universal compact quantum group $\QG_u$ acting on $\alg$ in a filtration preserving way. We call $\QG_u$ the quantum symmetry group of $(\alg, \omega, (V_i)_{i\in \Ind})$. The canonical action of $\QG_u$ on $\alg$ is faithful.
\end{tw}
\begin{proof}
Observe first that by Lemma \ref{isom} it suffices to show that the category $\catV^{alg}$  has a final object.
 Let us divide the proof into several steps.

\vspace*{0.5 cm}
\noindent I

In the first step we introduce some notation and construct an auxiliary algebra $\Dlg_V$.

Fix for each $i\in \Ind$  an orthonormal basis $\{e_1,\ldots,e_{k_i}\}$  in $V_i$ with respect to the scalar product given by the state $\omega$ (i.e.\ $\omega(e_l^*e_m)=\delta_{lm}1$, $l,m=1,\ldots,k_{i}$) and let $\{f_1,\ldots,f_{k_i}\}$  in $V_{i}^*$ be an orthonormal basis for $V_{i}^*$.
 Consider the family $(e_l^*)_{l=1}^{k_i}$ of elements of $V_{i}^*$. This family is linearly independent, so there exists an invertible matrix $S^{(i)}\in M_{k_i}$ such that
\begin{equation}\label{eq1} e_{l}^*=\sum_{m=1}^{k_i}  S^{(i)}_{lm} f_m \;\;\; l= 1,\ldots,k_i.\end{equation}
Put $Q_i=\overline{S^{(i)}}(S^{(i)})^T \in GL_{k_i}(\bc)$ and let $\dlg_{V}=\star_{i\in \Ind} A_u(Q_i)$, where for each $i\in \Ind$  the algebra $A_u(Q_i)=C_u(U^+(Q_i))$ is the universal unitary algebra of Van Daele and Wang (see \cite{univ}), with the canonical generating set $\{\mathcal{U}^{(i)}_{lm}:l,m=1,\ldots,k_i\}$. The algebra $\dlg_V$ is the algebra of continuous functions on the compact quantum group $\hat{\star}_{i\in \Ind} U^+(Q_i)$ (for more comments on the dual free product of compact quantum groups see for example \cite{bsk1}); the corresponding algebraic free product $R(\hat{\star}_{i\in \Ind} U^+(Q_i))$ will be denoted by $\Dlg_V$.

\vspace*{0.5 cm}
\noindent II

In the second step we show that if $\alpha_0$ is an algebraic action of a quantum group $\QH$ on $\Alg$ and $(\alpha_0,\QH)\in\catV^{alg}$, then the restriction of $\alpha_0$  to a map on $V_i$ determines in a natural way a representation of $\QH$. We also prove that this representation is automatically unitary and construct a $^*$-homomorphism from $\Dlg_V$ to $R_{\alpha}(\QH)$.

Let $(\alpha_0,\QH)\in \catV^{alg}$. Fix $i\in \Ind$ (and skip it from most of the notation in the next paragraph). Condition (iii) in Definition \ref{defcatalg} implies that there exists a matrix $U=(U_{lm})_{l,m=1}^k \in M_k(R(\QH))$
such that
\begin{equation}\label{eq2} \alpha_0(e_l) = \sum_{m=1}^k e_m \ot U_{ml}, \;\;\; l= 1,\ldots,k.\end{equation}
Due to the condition \eqref{statepreserve} $U$ is an isometry; indeed,
\begin{align*}
\delta_{lm} 1_{C_r(\QH)} &= \omega(e_l^*e_m) 1_{C_r(\QH)} = (\omega \ot 1_{C_r(\QH)}) (\alpha_0(e_l)^* \alpha_0(e_m)) \\&=
(\omega \ot 1_{C_r(\QH)})\big( (\sum_{p=1}^k e_p^* \ot U_{pl}^*)  \sum_{p=1}^k e_q \ot U_{qm}) \big) =
\sum_{p,q=1}^k \omega(e_p^* e_q)  U_{pl}^* U_{qm} \\&= \sum_{p=1}^k  U_{pl}^* U_{pm} =(U^*U)_{lm}.
\end{align*}
To show that $U$ is actually a unitary, we need to employ condition (ii) in Definition \ref{defcat} (it is easier here to use Lemma \ref{isom} and pass to the `analytic' version of $\alpha_0$, to be denoted by $\alpha$). Suppose that $U$ is not unitary. Viewing $U$ as an operator on the Hilbert module $\mathcal{M}:= \bc^k \ot C_r(\QH)$, we see that (by Theorem 3.5 in \cite{Lance}) there must exist some element in $\mathcal{M}$ which is not in the range of $U$; in fact its distance from the range of $U$ must be strictly greater than some $\epsilon>0$. In other words there is a sequence $b_1,\ldots,b_k$ of elements in $C_r(\QH)$ such that for all possible sequences $c_1, \ldots, c_k$ of elements in $C_r(\QH)$ we have
\[ b_l \neq \sum_{m=1}^k U_{lm} c_m, \;\; \textup{ for some } l\in \{1,\ldots,m\}.\]
Consider now an element $\mathbf{b}= \sum_{l=1}^k e_l \ot b_l \in V_i \odot C_r(\QH) \subset \alg \ot C_r(\QH)$. The last displayed formula means precisely that $\mathbf{b} \notin \alpha(V_i) (1\ot C_r(\QH))$. Moreover, the remark on the Hilbert module distance means that if $\mathbf{d}\in\alpha(V_i) (1\ot C_r(\QH))$ then \[\|(\omega \ot \id_{C_r(\QH)})((\mathbf{b} - \mathbf{d})^*(\mathbf{b} - \mathbf{d}))\|>\epsilon^2.\]
Consider then any $c\in \tu{Lin}\, \bigcup_{j\in \Ind} V_j$, say $c=\sum_{j\in F} c_j$, where $F$ is a finite subset of $\Ind$ and any family $(b_j)_{j \in F}$ of elements of $C_r(\QH)$.
Put  $\mathbf{d}:=\sum_{j\in F} \alpha(c_j)(1 \ot b_j)$, $\mathbf{d}_i=\alpha(c_i)(1 \ot b_i)$. Then
\[ \|\mathbf{b}-\mathbf{d}\|^2 \geq \|(\omega\ot \id_{C_r(\QH)}) ((\mathbf{b} - \mathbf{d})^*(\mathbf{b} - \mathbf{d}))\| \geq \|(\omega \ot \id_{C_r(\QH)}) ((\mathbf{b} - \mathbf{d}_i)^*(\mathbf{b} - \mathbf{d}_i))\|>\epsilon^2.\]
It follows from this that $\mathbf{b} \notin \alpha(\alg) (1\ot C_r(\QH))$, which is a contradiction.

Consider now $V_i^*$. As $\alpha_0$ preserves also this set, the above proof shows that the matrix $W=(W_{ml})_{m,l=1}^k\in M_k(R(\QH))$ determined by the condition
\begin{equation} \label{eq3} \alpha(f_l) = \sum_{m=1}^k f_m \ot W_{ml}, \;\;\; l= 1,\ldots,k,\end{equation}
is also unitary. A comparison of the formulas \eqref{eq1}-\eqref{eq3} yields the following equality:
\[ WS^{T} = S^T \bar{U},\]
so that the unitarity of $W$ transforms into the following condition:
\[ I= \bar{S}^{-1} \bar{U}^* \bar{S} S^T \bar{U} (S^{T})^{-1} =  S^T \bar{U} (S^{T})^{-1} \bar{S}^{-1} \bar{U}^* \bar{S},\]
or, putting $Q=\bar{S}S^T \in GL_k(\bc)$
\[I = \bar{U}^* Q\bar{U} Q^{-1} =\bar{U} Q^{-1} \bar{U}^*Q. \]

This means that the family $(U_{lm})_{l,m=1}^k$ satisfies the defining relations for the generators of Van Daele's and Wang's universal unitary algebra $A_u(Q_i)$. Hence there exists a unique unital $^*$-homomorphism $\pi_i: A_u(Q_i)\to C_r(\QH)$ such that
\begin{equation} \pi_i(\mathcal{U}_{lm}) = U_{lm}\in R(\QH) \label{piform}\end{equation}
for $l,m=1,\ldots,k$. It is easy to see that $\pi_i$ intertwines respective coproducts; moreover $\pi_i$ maps the $^*$-algebra $\Alg_u(Q_i)$ spanned by the elements of type $\mathcal{U}_{lm}$ into $R(\QH)$ (and even more specifically into $R_{\alpha}(\QH)$). Consider the algebraic free product of all the respective corestrictions of morphisms $\pi_i$: \[\pi_{\alpha,\QH} = \star_{i\in \Ind}\, \pi_i:\Dlg_{V}\to R_{\alpha}(\QH).\]
Note that the image of $\pi_{\alpha,\QH}$ is actually equal to  $R_{\alpha}(\QH)$.  

\vspace*{0.5 cm}
\noindent III

In the third step we introduce another class of $^*$-homomorphisms which generalise actions in $\catV^{alg}$ and establish some formulas satisfied by these homomorphisms.

In the rest of the proof we will only consider algebraic actions and denote them simply by $\alpha$.
We need to consider a larger class of $^*$-homomorphisms from $\Dlg_V$ into algebras of functions on compact quantum groups. This idea comes from \cite{Sabbe}.
Denote the collection of all finite sequences  $(\alpha_1,\QH_1), (\alpha_2,\QH_2),\cdots, (\alpha_k, \QH_k) \in \catV^{alg}$ ($k\in \bn$) by $T_{\cat}$. For each  such sequence $T\in T_{\cat}$ consider the $^*$-homomorphism $\alpha_T:\Alg \to \Alg \odot R(\QH_1)\odot \cdots \odot R(\QH_k)$ defined by
\[ \alpha_T=(\alpha_1 \ot \id_{R(\QH_2)} \ot \cdots \ot \id_{R(\QH_k)}) (\alpha_{k-1} \ot \id_{R(\QH_k)}) \alpha_k\]
($\alpha_T$ should be thought of as reflecting the composition of consecutive actions of $\QH_1, \cdots, \QH_k$ on $\Alg$ -- note however it need not be an action of the group $\QH_1 \times \cdots \times \QH_k$). Similarly for a sequence $T$ as above consider the mapping $ \pi_T:\Dlg_V \to R(\QH_1) \odot \cdots \odot R(\QH_k)$ given by
\[\pi_T=(\pi_{\alpha_1,\QH_1} \ot \cdots \ot \pi_{\alpha_k,\QH_k})\circ  \Com_{k-1},\]
where $\Com_k:\Dlg_V \to \Dlg_V^{\ot k}$ is the usual iteration of the coproduct of $\Dlg_V$ (and $\Com_0:=\id_{\Dlg_V}$).
Note that if $T,S\in T_{\cat}$ and $TS$ denotes the concatenation of the sequences, we have formulas
\begin{equation}  \alpha_{TS} = (\alpha_T \ot \id) \alpha_S, \label{alphaTS} \end{equation}
\begin{equation} \label{piTS} \pi_{TS} = (\pi_T \ot \pi_S)\circ\Com.\end{equation}


Define a linear map $\beta:\Alg \to \Alg \odot \Dlg_V$ via the linear extension of the formula (considering separately each $i\in \Ind$ and the orthogonal basis $e_1,\ldots, e_{k_i} \in V_i$):
\[ \beta(e_l) = \sum_{m=1}^{k_i} e_m \ot \mathcal{U}_{ml}, \;\;\; l=1, \ldots, k_i.\]
Observe that although $\beta$ need not be a $^*$-homomorphism, it is unital and moreover is a coalgebra morphism:
\begin{equation} \label{betacoalg}
(\beta \ot \id_{\Dlg_V}) \beta = (\id_{\Alg} \ot \Com) \beta
 \end{equation}
(it is enough to check the above equality on all the elements $e_l$, where it is elementary). Moreover we have
\begin{equation} \label{betadens} \beta(\Alg) (1 \odot \Dlg_V) = \Alg \odot \Dlg_V;\end{equation}
indeed, it is enough to show that the left hand side contains any element of the form $e_l^{(i)} \odot 1$, where $e_l^{(i)}$ is one of the basis elements of $V_i$. The latter elements can be obtained from the expressions of the type
\[  \sum_{m=1}^{k_i} \beta(e_m^{(i)}) (\mathcal{U}_{lm}^{(i)})^*.\]

Further we have for each $T\in T_C$
\begin{equation} \label{intertwalphapi}  \alpha_T = (\id_{\Alg} \ot \pi_T) \circ \beta.\end{equation}
Indeed, if the length of the sequence $T$ is $1$, then the formula above follows directly from the definition of $\pi_{\alpha,\QH}$ for $(\alpha, \QH)\in \catV^{alg}$.
Further, for any two sequences $T,S \in T_C$ for which \eqref{intertwalphapi} holds we have (using \eqref{alphaTS}, \eqref{piTS} and \eqref{betacoalg})
\begin{align*}
\alpha_{TS} &= (\alpha_T \ot \id)\alpha_S = (((\id_{\Alg} \ot \pi_T) \circ \beta) \ot \id)\circ (\id_{\Alg} \ot \pi_S) \circ \beta \\&=
(\id_{\Alg} \ot \pi_T \ot \pi_S)\circ (\beta \ot \id_{\Dlg_V}) \circ \beta = (\id_{\Alg} \ot \pi_T \ot \pi_S)\circ (\id_{\Alg} \ot \Com) \beta
\\&= (\id_{\Alg} \ot \pi_{TS})\Com,
\end{align*}
so \eqref{intertwalphapi} follows by induction for all sequences in $T_C$.

\vspace*{0.5 cm}
\noindent IV

Here we define the compact quantum group $\QG$ which will turn out to be our universal object.

Let $I_0=\bigcap_{T\in T_{\cat}} \tu{Ker}\, \pi_{T}$ (the class of objects in $\catV^{alg}$ need not be a set, but we can get around this problem in the usual way, identifying isomorphic objects and bounding the dimension of the algebras considered). Then $I_0$ is a two-sided $^*$-ideal in $\Dlg_V$. We will show that it is also a Hopf $^*$-ideal, i.e.\ that if $q: \Dlg_{V} \to  \Dlg_{V}/I_0$ is the canonical quotient map, then $(q \ot q)\Com (I)=\{0\}$. To this end it suffices (via the usual application of slice functionals) to show that if $S, T$ are sequences in $T_{\cat}$  then for each $b\in I$ we have
\begin{equation} \label{Hopfideal} (\pi_T \otimes \pi_S) \Com (b) =0.\end{equation}
This however follows from \eqref{piTS}. 
Thus the unital $^*$-algebra $\Dlg_V/_{I_0}$ is in fact a CQG algebra (in the terminology of \cite{DiK}). Denote the corresponding compact quantum group by $\QG$ (so that
$R(\QG)=\Dlg/_{I_0}$) and the quotient $^*$-homomorphism from $\Dlg$ onto $\Dlg_V/_{I_0}$ by $q$. The construction above shows that
\begin{equation} (q \ot q)\circ  \Com_{\Dlg_V}  = \Com_{\QG} \circ q.\label{copintDG}
\end{equation}

\vspace*{0.5 cm}
\noindent V

In this step we show that the quantum group $\QG$ acts on $\Alg$ in a $\mathcal{V}$-preserving way.

Let $\alpha_u:\Alg \to \Alg \odot R(\QG)$ be given by
\begin{equation}\label{defalphau}  \alpha_u= (\id_{\Alg} \ot q) \circ \beta.\end{equation}
We want to show that $\alpha_u$ is a $^*$-homomorphism. To this end it suffices to show that if $a,b \in \Alg$ then
\[ \beta (a^*) - \beta(a)^* \in \Alg \odot I_0, \;\;\; \beta(ab) - \beta(a) \beta(b) \in \Alg \odot I_0,\]
or, in other words, that for all $T\in T_{\cat}$
\[  (\id_{\Alg} \ot \pi_T) (\beta (a^*) - \beta(a)^*)=0,\;\;\; (\id_{\Alg} \ot \pi_T) ( \beta(ab) - \beta(a) \beta(b))=0.\]
The above formulas are however equivalent (by the fact that $\id_{\Alg} \ot \pi_T$ is a $^*$-homomorphism and by \eqref{intertwalphapi}) to the formulas
\[  \alpha_T (a^*) - \alpha_T(a)^*=0,\;\;\; \alpha_T(ab) - \alpha_T(a) \alpha_T(b)=0,\]
which are clearly true as each $\alpha_T$ is defined as a composition of $^*$-homomorphisms.
The fact that $\alpha_u$ satisfies condition (ii) in Definition \ref{defcatalg} follows by putting together \eqref{defalphau}, \eqref{betacoalg} and \eqref{copintDG}. Condition (iii) in Definition \ref{defcatalg} can be checked directly. Finally the nondegeneracy condition (i) is a consequence of \eqref{defalphau}, \eqref{betadens} and the fact that $q:\Dlg \to R(\QG)$ is a surjective homomorphism.

Hence $(\alpha_u, \QG)\in \catV^{alg}$. The fact that the action $\alpha_u$ is faithful follows from the construction.

\vspace*{0.5 cm}
\noindent VI

Finally we show that the pair $(\alpha_u, \QG)$ is the final object in $\catV^{alg}$.

Consider any object $(\alpha, \QH)$ in $\catV^{alg}$. Recall the map $\pi_{\alg,\QH}:\Dlg_V \to R(\QH)$. The kernel of $\pi_{\alg,\QH}$ is contained in $I_0$; hence there exists a unique map $\pi':\Dlg_V/_{I_0}\to R(\QH)$ such that $\pi_{\alg,\QH} =\pi' \circ q$. Using the fact that $\pi_{\alg,\QH}$ intertwines the coproducts of $\Dlg_V$ and $R(\QH)$ together with the formula \eqref{copintDG} we obtain that  $\pi':R(\QG)\to R(\QH)$ is a morphism of compact quantum groups. Similarly we compute
\[ (\id_{\Alg} \ot \pi')\alpha_u = (\id_{\Alg} \ot \pi')(\id_{\Alg} \ot q)\beta = (\id_{\Alg} \ot \pi_{\alpha,\QH})\circ \beta = \alpha,\]
where the last equality  follows from \eqref{intertwalphapi}. Thus $\pi'$ is a desired morphism in $\catV^{alg}$ between $(\alpha, \QH)$ and $(\alpha_u, \QG)$. Its uniqueness can be easily checked using the fact that the elements of the type $q(\mathcal{U}_{lm}^{(i)})$, $i\in \Ind$, $l,m =1,\ldots,k_i$ generate $R(\QG)$ as a $^*$-algebra.
\end{proof}

\begin{rem}
As mentioned before, the first part of the proof is inspired by the arguments in Section 4 of \cite{gos}. Here however we avoid any references to the Dirac operator and work directly with the filtration of the underlying $C^*$-algebra, which makes this framework in principle more general (see the beginning of the next section). The proof above is also self-contained and purely algebraic, which is possible due to Lemma \ref{isom}, and leads to a simplification of  certain technical arguments. Note however that this algebraic approach does not seem to be suitable for working with more general notion of \emph{quantum families of maps}, considered in \cite{gos}.
\end{rem}

The following corollary corresponds to the obvious classical fact that each group acting faithfully on a given structure can be viewed as a subgroup of the full symmetry group of this structure.

\begin{cor} \label{faithfulsubgroup}
Let $(\alg, \omega, (V_i)_{i\in \Ind})$ be a $C^*$-algebra with an orthogonal filtration and let $(\alpha_u,\QG_u)$ be the universal object in $\catV$.
If $(\alpha,\QG) \in \catV$ and the action $\alpha$ is faithful, then the morphism $\pi_{\alpha}:R(\QG_u)\to R(\QG)$ constructed in Theorem \ref{qsymV}  is surjective. In other words, $\QG$ is a quantum subgroup of $\QG_u$.
\end{cor}

\begin{proof}
It suffices to observe that it follows from the construction in the proof of Theorem \ref{qsymV} that the image of the morphism $\pi_{\alpha}$ contains $R(\QG)_{\alpha}$.
\end{proof}



Below we record certain special cases in which some properties of the universal quantum symmetry group $\QG_u$ follow directly from certain properties of the filtration.

\begin{tw} \label{gentrace}
Let $(\alg, \omega, (V_i)_{i\in \Ind})$ be a $C^*$-algebra with an orthogonal filtration and let $\QG$ be its quantum symmetry group, with a corresponding action $\alpha:\alg \to \alg \ot C(\QG)$. The following implications hold:
\begin{rlist}
\item if $\omega$ is a trace then $\QG_u$ is a compact quantum group of Kac type;
\item if there exists a finite set $F\subset \Ind$ such that the union of subspaces $\bigcup_{i\in F} V_i$  generates $\alg$ as a $C^*$-algebra, then $\QG_u$ is a compact matrix quantum group;
\item if there exists $i\in \Ind$ such that $V_i$ generates $\alg$ as a $C^*$-algebra and $\{e_{1},\cdots,e_k\}$ is an orthonormal basis of $V_i$ with respect to the scalar product determined by $\omega$ (so that $\omega(e_l^* e_m)= \delta_{lm} 1$ for $l,m=1,\ldots,k$), then the matrix  $U=(U_{lm})_{l,m=1}^k$ of elements of $C(\QG)$ determined by the condition
\[ \alpha(e_l) = \sum_{m=1}^k e_m \ot U_{ml}, \;\;\; j= 1,\ldots,k,\]
 is a fundamental unitary representation of $\QG$ (and $\bar{U}$ is also unitary). In particular $\QG$ is a quantum subgroup of $U^+_k$.
\end{rlist}

\end{tw}
\begin{proof}
It suffices to look at the proof of Theorem \ref{qsymV} and note that if $\omega$ is a trace and $\{e_1,\ldots,e_k\}$ is an orthonormal basis of $V_i$ then $\{e_1^*,\ldots,e_k^*\}$ is an orthonormal basis of $V_i^*$, so that the matrix $Q_i$ appearing in that proof is equal to $I_k$.
\end{proof}

We finish the section by discussing a functorial property of the construction of the quantum symmetry groups given in Theorem \ref{qsymV}.

Let $\alg$ be a unital $C^*$-algebra equipped with a faithful state $\omega$ and two orthogonal (with respect to $\omega$) filtrations $\mathcal{V}$ and $\mathcal{W}$, with indexing sets respectively $\mathcal{I}$ and $\mathcal{J}$. We say that $\mathcal{W}$ is a \emph{subfiltration} of $\mathcal{V}$ if for each $j \in \Jnd$ there exists $i \in \Ind$ such that $W_j\subset V_i$. We have the following  corollary of Theorem \ref{qsymV}.

\begin{cor} \label{generatedfiltrations}
Let $\alg$ be a unital $C^*$-algebra equipped with a faithful state $\omega$ and two orthogonal (with respect to the same state $\omega$) filtrations $\mathcal{V}$ and $\mathcal{W}$, such that $\mathcal{W}$ is a subfiltration of $\mathcal{V}$. 
Denote the respective quantum symmetry groups by $\QG_{\mathcal{V}}$ and $\QG_{\mathcal{W}}$. 
Then $\QG_{\mathcal{W}}$ is a quantum subgroup of   $\QG_{\mathcal{V}}$.
\end{cor}
\begin{proof}
It suffices to observe that $\QG_{ \mathcal{W}}$ with its canonical action $\alpha$ on $\alg$ is an object in $\catV$ and that, due to Theorem \ref{qsymV}, the action $\alpha$ is faithful. Corollary \ref{faithfulsubgroup} ends the proof.
\end{proof}

It is easy to check that if given two orthogonal filtrations of a fixed $C^*$-algebra $\alg$ (with respect to the same state $\omega$) we define the family of sets $\mathcal{V}\vee \mathcal{W}$ as consisting of all sets of the type $V_i \cap W_j, i \in \mathcal{I}, j \in \mathcal{J}$, we obtain another orthogonal filtration of $\alg$, which can be thought of as the filtration generated by $\mathcal{V}$ and $\mathcal{W}$ (we can work with a smaller indexing set $\{(i,j) \in \mathcal{I} \times \mathcal{J}: V_i \cap W_j \neq \emptyset\}$ and identify $(0,0)\in I\times J$ as its distinguished element). By the above corollary $\QG_{\mathcal{V}\vee \mathcal{W}}$ is a quantum subgroup of  both  $\QG_{\mathcal{V}}$ and $\QG_{\mathcal{W}}$.


\section{Algebras with orthogonal filtrations and related quantum symmetry groups}

In this section we first explain when the quantum symmetry groups associated with $C^*$-algebras equipped with orthogonal filtrations can be viewed as quantum isometry groups of noncommutative manifolds, as introduced in \cite{gos} and \cite{JyotDeb2}. Then we present examples of orthogonal filtrations arising in the finite-dimensional setting, in the context of Cuntz-Krieger algebras and when considering ergodic compact (quantum) group actions; in each case we briefly discuss resulting quantum symmetry groups.

\subsection{Relations to the spectral triple framework}

It is natural to ask under what circumstances quantum symmetry groups defined in the previous section can be put directly into the framework of quantum isometry groups considered in \cite{gos} and \cite{JyotDeb2}. For that we need the following definition.

\begin{deft} \label{Diracfilt}
Let  $(\alg, (V_m)_{m=0}^{\infty})$ be a $C^*$-algebra with an orthogonal filtration. We say that this filtration is of spectral triple type if
 for each $k\in \bn$ there exists $M\in \bn$ such that for all $l \in \bn$
\begin{equation} \label{inv1} V_k V_l \subset \bigcup_{j=0}^{l+M} V_j, \;\;\;\;\;  V_k^* V_l \subset \bigcup_{j=0}^{l+M} V_j.\end{equation}
\end{deft}

The second condition in the definition above follows from the first one if for each $i\in \Ind$ there exists $i^*\in \Ind$ such that $V_i^*=V_{i^*}$. This is the case for all examples studied in this paper.

\begin{tw} \label{Diracequiv}
Let  $(\alg, (V_m)_{m=0}^{\infty})$ be a $C^*$-algebra with an orthogonal filtration of spectral triple type. Then there exists a spectral triple $(\Alg,D,\Hil)$ on $\alg$ such that the category $\catV$ coincides with the category of quantum groups of isometries of the triple $(\Alg,D,\Hil)$ considered in \cite{JyotDeb2}. In particular the quantum symmetry group of $(\alg, (V_m)_{m=0}^{\infty})$ coincides with the quantum isometry group of $(\Alg,D,\Hil)$ constructed in \cite{JyotDeb2}.
\end{tw}

\begin{proof}
Let $(\pi,\Hil,\Omega)$ be the GNS triple corresponding to the state $\omega$, and let $\iota$ denote the canonical embedding of $\alg$ into $\Hil$. The spaces $\iota(V_n)$ are finite-dimensional, mutually orthogonal subspaces of $\Hil$. Denote the corresponding orthogonal projections onto $\Hil_n$ by $P_n$ and put $D=\sum_{n=0}^{\infty} nP_n$. Then $D$ is a densely defined, essentially selfadjoint operator on $\Hil$ with compact resolvent. If $k\in \bn, a \in V_k$, then the operator $\pi(a)$ maps the space $\Hil_0:=\bigcup_{n\in \bn} \iota(V_n)$ into itself and moreover on this subspace
\[ [D,\pi(a)] = \sum_{n,m=0}^{\infty} [D, P_n \pi(a) P_m] = \sum_{n,m=0}^{\infty} (n-m) P_n \pi(a) P_m.\]
The first condition in \eqref{inv1} implies that $P_n \pi(a) P_m =0$ if $n>m+M$, and the second condition in \eqref{inv1} implies that $P_n \pi(a) P_m =0$ if $m>n+M$, so that in fact the sum above reduces to
\[  [D,\pi(a)] = \sum_{n,m\in \bn_0, |n-m|\leq M}  \sum_{n,m=0}^{\infty} (n-m) P_n \pi(a) P_m\]
and the commutator is bounded. Hence $(\Alg,D,\Hil)$ is a spectral triple on $\alg$. It is now straightforward to check that this triple fits into the framework considered in Section 2 of \cite{JyotDeb2}; moreover the category $\widehat{\mathbf{C}}$ introduced in Definition 2.26 of that paper can be easily seen to coincide with category $\catV$ studied here. This ends the proof.
\end{proof}

\begin{rem}
Note that the theorem above provides only sufficient conditions for the quantum symmetry group of an orthogonal filtration to arise as a quantum isometry group of a spectral triple. In general it seems very difficult to show that for a given orthogonal filtration the quantum symmetry group does not fit into the framework of \cite{JyotDeb2}; it remains however true that in many examples (see the rest of the paper) orthogonal filtrations arise in a natural way, whereas the (noncommutative) geometric structure is not apparent.
\end{rem}

\subsection{Finite-dimensional case}

If $\alg$ is a finite-dimensional algebra equipped with an orthogonal filtration $\mathcal{V}$, then we are in the spectral triple framework (by Theorem \ref{Diracequiv}). Moreover, the corresponding quantum symmetry group $\QG_{\mathcal{V}}$ is a quantum subgroup of the quantum automorphism group $\QG_{\textup{aut}}(\alg, \omega)$ studied in \cite{wa2}. In particular when $\alg =\bc^n$ for some $n \in \bn$, $\QG_{\mathcal{V}}$ is a quantum subgroup of the quantum permutation group $S^+_n$ (note that any compact quantum group acting on $\bc^n$ automatically preserves the canonical trace). In the latter case we can however be more specific, as described below.

Let $\alg=\bc^n$ be equipped with an orthogonal filtration $\mathcal{V}$ and let $\alpha:\bc^n \to \bc^n \ot C(\QH)$ be an action of a compact quantum group $\QH$ on $\bc^n$. Consider the \emph{magic unitary} (a unitary matrix whose entries are orthogonal projections) $U\in M_n (\QH)$ determined by
\[ \alpha(e_l) = \sum_{k=1}^n e_k \ot U_{kl}, \;\;\; l=1,\ldots,n\]
($e_1, \ldots, e_n$ denotes above the canonical basis of $\bc^n$). Let for each $i\in I$ the symbol $P_i$ denote the orthogonal projection in $B(\bc^n)$ onto the subspace $V_i$. Then it is easy to check that the action $\alpha$ preserves the filtration $\mathcal{V}$ if and only if the unitary $U$ commutes with each of the matrices $P_i$. We can replace this sequence of commutation conditions: if we enumerate the index set $\Ind$ (say from $1$ to $l$), the action $\alpha$ preserves the filtration $\mathcal{V}$ if and only if it commutes with the (self-adjoint) matrix $Q=\sum_{i=1}^l i P_i$.  Note that it is important in the above that we write down the matrix of $P_i$ in the canonical basis of $\bc^n$ (otherwise $U$ would still be unitary, but could cease to be a magic one). This implies that we are in a framework very similar to that of \cite{banmet}, where the quantum isometry groups of finite metric spaces were introduced, and the corresponding condition was expressed in terms of the commutation with the matrix given by the metric.

Consider a particular example.

\begin{example} \label{QISOZ4}
Let $n=4$, and for $k=0,1,2,3$ put $f_k=\sum_{l=1}^4 \exp(\frac{kl\pi}{2}) e_l$.
Let a filtration $\mathcal{V}$ of $\bc^4$ be given by $V_0=\{f_0\}, V_1=\{f_1,f_3\}, V_2=\{f_2\}$. By the arguments above the quantum isometry group $\QG_{\mathcal{V}}$ is a quantum subgroup of $S^+_4$, and the conditions on the magic unitary defining a fundamental unitary representation $U$ of $\QG_{\mathcal{V}}$ are given by the commutation with projections onto respective $V_i$'s. Here however the commutation with all three of these projections follows from the commutation with $P_2$ (as the commutation with $P_0$ is a consequence of the fact that $U$ is a unitary and $P_1=I-P_0-P_2$). As
\[P_2 = \frac{1}{4} \left(\begin{array}{cccc} 1 & -1 & 1 & -1 \\ -1 &1 &-1&1 \\ 1 & -1 & 1 & -1 \\ -1 &1 &-1&1\end{array}\right),\]
the fact that a magic unitary  $U$ commutes with $P_2$ is equivalent to the fact that it commutes with
\[ \left(\begin{array}{cccc} 0 & 0 & 1 & 0\\ 0 &0 &0&1 \\ 1 & 0 & 0 & 0 \\ 0 &1 &0&0\end{array}\right),\]
which is the adjacency matrix of a graph made of two segments. Now the Table in Section 7 of \cite{BanBic} implies that $\QG_{\mathcal{V}}$ is the free wreath product $\bz_2 \wr_* \bz_2$.
\end{example}

We will return to this example in Section \ref{group}.


\subsection{AF case - Christensen-Ivan spectral triples}
A natural extension of the finite-dimensional case is given by AF algebras, i.e.\ inductive limits of finite-dimensional algebras. These can be equipped with natural orthogonal filtrations of spectral triple type and hence also with natural spectral triples, see \cite{Chrivan}. The corresponding quantum isometry groups were studied in \cite{bgs}. We refer to that paper for details; here note only that the resulting quantum isometry groups are themselves inductive limits of quantum isometry groups of finite-dimensional algebras (see Theorem 1.2 of \cite{bgs}) and can be thought of as quantum symmetry groups of the associated Bratteli diagrams.

\subsection{Cuntz-Krieger algebras}

Let $n \in \bn$ and $A$ be an $n$ by $n$ matrix of $0$s and $1$s in which each row and column is non-zero. Recall that the Cuntz-Krieger algebra $\OA$ first defined in \cite{CK} is the universal $C^*$-algebra generated by partial isometries $S_1,\ldots,S_n$ such that for all $i,j=1,\ldots,n$, $i\neq j$,
\[ S_i^*S_j=0, \;\;\; S_i^*S_i = \sum_{k=1}^n A_{ik} S_k S_k^*.\]
In particular if $A$ is a matrix with all entries equal to $1$ we obtain the Cuntz algebra $\On$. We will use a standard multi-index notation for compositions of the generating partial isometries, so that if $\mu=(\mu_1,\ldots, \mu_k)$ is a sequence of indices with values in $\{1,\ldots,n\}$, we write $S_{\mu}=S_{\mu_1}\cdots S_{\mu_k}$.

The algebra $\OA$ is equipped with a  gauge action $\gamma:\bt\to\textup{Aut}(\OA)$ determined by the formula
\[ \gamma_t (S_i) = S_i, \;\;\; t\in \bt,  i=1,\ldots,n.\]
It is well-known (see Section 2 of \cite{CK}) that the fixed point subalgebra of $\OA$ for $\gamma$, denoted $\OA^{\gamma}$ is an AF algebra. Moreover the formula
\[ E(a)= \int_{\bt} \gamma_t(a) dt, \;\; a \in \OA\]
(with $dt$ denoting the normalised Lebesgue measure) is a faithful conditional expectation from $\OA$ onto $\OA^{\gamma}$.
Let $\tau$ be a trace on $\OA^{\gamma}$, put $\omega = \tau \circ E$ and consider a Christensen-Ivan type orthogonal filtration on $\OA^{\gamma}$ (see the previous subsection), say $\mathcal{W}=\{W_k:k\in \bn_0\}$. Define a filtration of $\OA$ in the following way: for each $k,m\in\bn_0$ let $V_{k,m}=\textup{Lin} \{S_\mu x, yS_{\nu}: |\mu| =|\nu|=m, x,y \in W_k\}$. Each subspace $V_{k,m}$ is finite-dimensional, together they span a dense $^*$-subalgebra of $\OA$ due to Lemma 2.2 in \cite{CK} (sometimes called a \emph{normal ordering} property) and if $(k,m)\neq (k',m')$, $a\in V_{k,m}, b\in V_{k',m'}$ then $\omega(a^*b)=0$. Indeed, if $m\neq m'$ then already $E(a^*b)=0$, and if $m=m'$, we note that $\omega(a^*b) = \tau(a^*b)$ and then can exploit the tracial property of $\tau$ and the orthogonality of the filtration $\mathcal{W}$ to reach the desired conclusion.

In fact the construction of the orthogonal filtration of $\OA$ proposed above contains a rather arbitrary choice of a filtration on the fixed point algebra. Given a faithful state $\sigma$ on $\OA^{\gamma}$, one can define a quantum symmetry group of a Cuntz-Krieger algebra in a different way, as a universal object for these compact quantum group actions on $\OA$ which preserve the state $\omega:=\sigma \circ E$ and the span of generating partial isometries, $\textup{Lin}\{S_1,\ldots,S_n\}$. We intend to provide the proof of the existence of such a  universal quantum symmetry group (which can be carried along the lines of that of Theorem \ref{qsymV}) and compare the two constructions in future work. Here we only observe that when the latter is applied to the Cuntz algebra $O_n$ with the canonical KMS-state $\omega =\tau \circ E$, where $\tau$ is the unique trace on the UHF-algebra $\On^{\gamma}$, the resulting quantum symmetry group is the universal quantum unitary group $U_n^+$ and its action on $\On$ is the one studied in \cite{Wang3}.

Suitable versions of this construction are in fact valid for general graph algebras (for the definition of the latter see \cite{Rae}).

\subsection{Orthogonal filtration determined by an ergodic action of a compact (quantum) group}

When a compact group $G$ acts on a unital $C^*$-algebra $\alg$, the action decomposes $\alg$ into \emph{spectral subspaces}, essentially corresponding to irreducible representations of $G$. The same fact remains valid for actions of compact quantum groups (\cite{Podles}, \cite{Boca}). 

\begin{deft}
An action $\alpha$ of a compact quantum group $\QG$ on a unital $C^*$-algebra $\alg$ is called ergodic if the fixed point subalgebra $\textup{Fix}\, \alpha=\{a \in \alg: \alpha(a) = a \ot 1_{C(\QG)}\}$ is one-dimensional.
\end{deft}

It is easy to check that if the action $\alpha$ is ergodic, then there is a unique state $\omega\in \alg^*$ which is preserved by $\alpha$ (it is given by the formula $\omega = (\id_{\alg} \ot h)\circ \alpha$, with the canonical identification $\textup{Fix}\, \alpha \approx \bc$ in mind). We say that the action is \emph{reduced} if the state $\omega$ is faithful. Note that if necessary, we can always pass to the reduced version of any given action, simply by considering a GNS representation of $\alg$ with respect to $\omega$ (see Section 2).

Given an $n$-dimensional irreducible representation $U$ of a compact quantum group  $\QG$ its character is defined by the usual formula $\chi_U:=\sum_{i=1}^n U_{ii}$. The representation $U$ has also a \emph{quantum dimension} $d_{U}\in [n,\infty)$ (\cite{wo1}). In fact both $d_U$ and $\chi_U$ depend only on the equivalence class of $U$, so we also write $\chi_{u}$ and $d_u$ for $u \in \textup{Irr} (\QG)$.  Define for each such $u$ a bounded functional $h_u$ on $C(\QG)$ by the formula
\[h_{u} (a) = d_{u} h ((\id \ot f_1)(\Com (\chi_{u}))     a), \;\; a \in C(\QG),\]
where $f_1$  denotes the \emph{Woronowicz character} of $\QG$ (see \cite{wo1}). In particular if $0$ denotes the trivial representation of $\QG$ we obtain $h_0=h$.

The following theorem can be read out from the results of \cite{Boca} (see also \cite{Podles} and \cite{Reiji}).

\begin{tw}[\cite{Boca}]\label{ergodicfiltration}
Let $\alpha:\alg \to \alg \ot C(\QG)$ be a reduced ergodic action of a compact quantum group $\QG$ on a unital $C^*$-algebra $\alg$ and let $\omega\in \alg^*$ be the unique invariant state for this action. Further let $\textup{Irr} (\QG)$ be the complete set of equivalence classes of irreducible unitary representations of $\QG$.
Define for each $u\in \textup{Irr} (\QG)$
\[P_{u} = (\id_{\alg} \ot h_{u}) \circ \alpha, \;\;\; V_u = P_u (\alg).\]
Then the collection $\{V_u: u \in \textup{Irr} (\QG), V_u\neq\{0\}\}$ is an orthogonal filtration of $\alg$ with respect to the state $\omega$. Moreover $(\alpha, \QG)\in \catV$.
\end{tw}

In view of the above theorem it is natural to ask when, given an ergodic action $\alpha$ of a compact quantum group $\QG$ on a unital $C^*$-algebra $\alg$ and the associated orthogonal filtration $\mathcal{V}$ constructed as above, $(\alpha, \QG)$ is the final object in $\catV$.

Below we answer this question in the simplest case, when $\QG=\hat{\Gamma}$ for a discrete group $\Gamma$. Recall that the irreducible representations of $\hat{\Gamma}$, and thus also spectral subspaces of actions of $\hat{\Gamma}$ are all $1$-dimensional, indexed by elements of $\Gamma$.

\begin{tw}
Let $\Gamma$ be a discrete group and let $\QG=\hat{\Gamma}$ (so that $C(\QG)=C^*(\Gamma)$). Suppose that $\alpha$ is a reduced ergodic action of $\hat{\Gamma}$ on a unital $C^*$-algebra $\alg$. Then the quantum symmetry group $\QG_{\mathcal{V}}$ associated with the orthogonal filtration $\mathcal{V}$ described in Theorem \ref{ergodicfiltration} is isomorphic to $\widehat{\Gamma'}$, where $\Gamma'$ is the subgroup of $\Gamma$ containing precisely these elements $\gamma\in \Gamma$ for which $V_{\gamma}\neq \{0\}$.
\end{tw}

\begin{proof}
Each of the irreducible representations of $\hat{\Gamma}$ has both the classical and the quantum dimension equal $1$. Hence, by Theorem 3.8 of \cite{Reiji}, the corresponding spectral subspaces $V_\gamma$ of $\alg$ are at most 1-dimensional. Let $\Gamma'$ denote the set of these $\gamma\in \Gamma$ for which $V_{\gamma}\neq \{0\}$. As given $\gamma_1, \gamma_2 \in \Gamma$ and $b_{\gamma_1} \in V_{\gamma_1}$, $b_{\gamma_2} \in V_{\gamma_2}$ there is $b_{\gamma_1}^* \in V_{\gamma_1^{-1}}$ and $b_{\gamma_1} b_{\gamma_2} \in V_{\gamma_1 \gamma_2}$ (and if $\gamma\in \Gamma'$ then $b_{\gamma}$ can be chosen unitary) $\Gamma'$ is indeed a subgroup of $\Gamma$.
Let $\omega\in \alg^*$ be the unique invariant state for the action $\alpha$ and choose for each $\gamma\in \Gamma'$ an element $b_{\gamma} \in V_{\gamma}$ such that
$\omega(b_{\gamma}^* b_{\gamma})=1$. Consider now an arbitrary element $(\alpha', \QH)\in \catV$ and assume that the action $\alpha'$ is faithful (as it is in the case of the universal action). It follows from the arguments in Section 2 that for each $\gamma\in \Gamma'$
\[ \alpha'(b_{\gamma}) = b_{\gamma} \ot z_{\gamma},\]
where $z_{\gamma}\in C(\QH)$ is a one-dimensional unitary representation of $\QH$ -- in other words a group-like unitary in $C(\QH)$. By the faithfulness assumption $C(\QH)$ is spanned by all $z_{\gamma}$, $\gamma\in \Gamma'$. Peter-Weyl theory for representations of compact quantum groups implies therefore that $\QH\approx \hat{K}$ for some discrete group $K$. Moreover each $z_{\gamma}$ is of the form $\lambda_{k_{\gamma}}$ for some $k_{\gamma}\in \Gamma$. It is easy to check that the map $\gamma \to k_{\gamma}$ is a surjective group homomorphism from $\Gamma'$ to $K$. As $\Gamma'$ is the universal discrete group with the property that such a homomorphism exists, a standard argument ends the proof.
\end{proof}

It follows from the above result that if a compact quantum group of the form $\hat{\Gamma}$ has an ergodic \emph{faithful} action on a  unital $C^*$-algebra $\alg$ and $\mathcal{V}$  is the orthogonal filtration constructed as in Theorem \ref{ergodicfiltration}, $(\alpha, \hat{\Gamma})$ is the final object in $\catV$. This is no longer true for faithful ergodic actions of general compact quantum groups, as the following example shows.

\begin{example}
Let $\alpha$ denote the action of the permutation group $S_3$ on itself by the left multiplication. It is clearly ergodic and faithful, with the invariant state given by the (rescaling of) the counting measure. As $\alpha$ corresponds to the left regular representation, we know that its decomposition into spectral subspaces is determined just by the representation theory of the relevant group. Recall that $S_3$ has three irreducible representations: the trivial one, another one-dimensional representation given by the signum and a two-dimensional representation. Denote $\alg=C(S_3)\approx \bc^6$, let $\pi_1, \ldots, \pi_6$ be an enumeration of permutations in $S_3$ such that $\tu{sgn}\,\pi_i=1$ for $i=1,2,3$,  $\tu{sgn}\,\pi_i=-1$ for $i=4,5,6$, and write $e_i:=\delta_{\pi_i}\in C(S_3)$, $f=e_1+e_2+e_3 - e_4 -e_5 -e_6$. By the general properties of left regular representation, spectral subspaces of $\alg$ are then given by $V_0=\bc 1$, $V_1=\bc f$, $V_2=\alg \ominus (V_0 \oplus V_1)$ (where the orthogonal complement refers to the state of $\alg$ coming from the counting measure. Recapitulating, the category $\catV$ consists of these quantum subgroups of $S_6^+$ which preserve the one-dimensional subspace $\bc f$ (see Example \ref{QISOZ4} for a similar argument). Writing the action $\alpha$ in the form
\[ \alpha(e_j) = \sum_{i=1}^6 e_i \ot p_{ij}, \;\; j=1,\ldots,6,\]
(where $(p_{ij})_{i,j=1}^6$  is a magic unitary), the condition of preserving $\bc f$ can be written (after some computations) as follows:
\begin{align*} p_{11}+p_{21}+p_{31} &= p_{12} + p_{22} + p_{32} = p_{13} + p_{23} + p_{33} \\&= p_{44}+p_{45}+p_{46} = p_{54} + p_{55} + p_{56} = p_{64} + p_{65} + p_{66} \end{align*}
Thus the dual of the free product $\bz_2 \star \bz_2=D_{\infty}$ with a faithful action on $\alg\approx \bc^6$ given by the matrix
\[ \left(\begin{array}{cccccc} p & p^{\perp} & 0 & 0 & 0 & 0 \\ p^{\perp} & p & 0 & 0 & 0 & 0  \\ 0 &0 & 1 & 0 & 0 &0\\ 0 & 0 & 0 & q & q^{\perp} & 0
\\ 0 & 0 & 0 & q^{\perp} & q & 0 \\ 0 &0 & 0 & 0 & 0 &1
\end{array}\right),\]
where $p,q$ are free projections generating $C^*(\bz_2 \star \bz_2)$, belongs to $\catV$. Hence the final object in $\catV$ contains $\widehat{D_{\infty}}$ as a quantum subgroup, so cannot coincide with $S_3$.
\end{example}

Finally note that in general the action of a quantum symmetry group of a $C^*$-algebra $\alg$ equipped with an orthogonal filtration need not be ergodic. It is enough to consider $\alg=\bc^3$ with the canonical trace and the filtration given by $V_0=\bc 1$, $V_1=\textup{Lin} (1,-1,0)$, $V_2=\textup{Lin} (1,1,-2)$ (this filtration corresponds to the one induced by a graph with three vertices and a single edge). Then the quantum symmetry group $\QG_{\mathcal{V}}$ is a standard group $\bz/_{2\bz}$ acting by permuting the first two coordinates in $\bc^3$ and its action is clearly not ergodic. Another example can be found in Proposition 3.1 of \cite{Wang3}.

\section{Filtrations of group $C^*$-algebras and the quantum group actions preserving the word length or the block length}\label{group}

In this section we explain how the results of Section 2 are related to quantum isometry groups of the group duals. We compute the quantum isometry group of the dual of $\bz_2^{\oplus k}$ and further show how the new framework introduced in Section 2 provides in particular the language suitable for the discussion of the universal quantum groups acting on the dual of a free product of finitely generated discrete groups and preserving the block length. We also indicate in a short final subsection how the framework of quantum symmetry groups of $C^*$-algebras equipped with orthogonal filtrations allows us to extend the concepts developed in \cite{bhs} and further studied in \cite{bsk1} and \cite{bsk2} from duals of discrete groups to general compact quantum groups.

\subsection{Finitely generated groups and the usual length}

Let $\Gamma$ be a discrete group. The elements of the reduced group $C^*$-algebra will be denoted in the same way as the elements of $\Gamma$; in particular we identify the group ring $\bc[\Gamma]$ as a subalgebra of $C^*_r(\Gamma)$ via $\bc[\Gamma]=\textup{span}\{\gamma:\gamma \in \Gamma\}.$ The canonical trace on $C^*_r(\Gamma)$ is given by the continuous extension of the formula:
\[\tau(\gamma) = \left\{ \begin{array}{ccc} 1 & \textup{ if } & \gamma=e \\ 0 & \textup{ if } & \gamma\neq e \\ \end{array} \right.\]

We will consider below partitions of $\Gamma$ into finite sets, always assuming that $\{e\}$ (where $e$ denotes the neutral element of $\Gamma$)  is one of the sets in the partition. The following lemma is straightforward.


\begin{lem}\label{filtGamma}
If $\mathcal{F}=(F_i)_{i\in \Ind}$ is a partition of $\Gamma$ into finite sets and $V^{\mathcal{F}}_i:=\textup{span}\{\gamma:\gamma\in F_i\}\subset C^*_r(\Gamma)$ ($i\in \Ind$), then
the pair $( \tau, (V^{\mathcal{F}}_i)_{i\in \Ind})$ defines an orthogonal filtration of $C^*_r(\Gamma)$.
\end{lem}

\begin{deft} \label{qsymF}
The quantum symmetry group of $(C^*_r(\Gamma), \tau, (V^{\mathcal{F}}_i)_{i\in \Ind})$, defined according to Theorem \ref{qsymV}, will be called the quantum symmetry group of $\widehat{\Gamma}$ preserving the partition $\mathcal{F}$.
\end{deft}

For a discrete group $\Gamma$ and any vector space $\mathcal{V}$ we will consider the linear maps $f_{\gamma}: \mathcal{V} \odot \bc[\Gamma]\to \mathcal{V}$ ($\gamma\in \Gamma$) defined by the linear extension of the prescription
\[ f_\gamma (\gamma' \ot v) = \delta_{\gamma,\gamma'} v, \;\;\; \gamma' \in \Gamma, v \in \mathcal{V}.\]

\begin{deft} \label{presergendef}
Let $\QG$ be a compact quantum group and assume that $\alpha:C^*_r(\Gamma) \to C^*_r(\Gamma) \ot C(\QG)$ is an action of $\QG$ on $C^*_r(\Gamma)$. Let $\mathcal{F}=(F_i)_{i\in \Ind}$ be a partition of a discrete group $\Gamma$ into finite sets. The action $\alpha$ is said to preserve  $\mathcal{F}$ if
\begin{rlist}
\item $\alpha: \bc[\Gamma] \to \bc[\Gamma] \odot C(\QG)$;
\item for all $i\neq j \in \Ind$ and  $\gamma \in F_i$ there is $f_{\gamma}|_{\alpha(F_j)} = 0$.
\end{rlist}
\end{deft}

\begin{prop} \label{equivfiltpart}
Let $\mathcal{F}=(F_i)_{i\in \Ind}$ be a partition of a discrete group $\Gamma$ into finite sets. The quantum symmetry group of $\widehat{\Gamma}$ preserving the partition $\mathcal{F}$ is the universal compact quantum group $\QG$ acting on $C^*_r(\Gamma)$ via an $\mathcal{F}$ preserving action.
Moreover if for some $i\in \Ind$ the subset $F_1=\{\gamma_1,\ldots,\gamma_k\}$ generates $\Gamma$ as a group, then the matrix $U=(U_{lm})_{l,m=1}^k\in M_k (C(\QG))$ given by
\[ \alpha(\gamma_m) = \sum_{l=1}^k \gamma_l \ot U_{lm}, \;\;\; m=1,\ldots,k,\]
is a fundamental unitary representation of $\QG$.
\end{prop}
\begin{proof}
The first part of the proposition follows from the comparison of the conditions defining respective classes of quantum group actions (the one in Definition \ref{defcat} and the one in Definition \ref{presergendef}). The second is a consequence of Theorem \ref{gentrace}.
\end{proof}

\begin{tw}
Let $\Gamma$ be a finitely generated  group and let $l:\Gamma \to \bn_0$ denote the word-length function with respect to a fixed finite generating set. Put, for each $m\in \bn_0$,
$F_m=\{\gamma\in \Gamma:l(\gamma)=m\}$. Then $\mathcal{F}=(F_m)_{m=0}^{\infty}$ is a partition of $\Gamma$ into finite sets, corresponding filtration of $C^*_r(\Gamma)$ is of spectral triple type and the quantum symmetry group of $\widehat{\Gamma}$ preserving the partition $\mathcal{F}$  coincides with the quantum isometry group $\QISO(\widehat{\Gamma})$ defined in Section 2 of \cite{bhs} (and denoted there by $\QISO^+(\widehat{\Gamma})$).
\end{tw}

\begin{proof}
The fact that the filtration of $C^*_r(\Gamma)$ associated with $\mathcal{F}$ is of spectral triple type is a consequence of the subadditivity of the length function. Consider the construction of the spectral triple provided in the proof of Theorem \ref{Diracequiv}. The action of the resulting Dirac operator on the basis vectors in $l^2(\Gamma)$ is given by the formula $D(\delta_{\gamma}) = l(\gamma) \delta_{\gamma}$ ($\gamma \in \Gamma$). This is precisely the Dirac operator studied in \cite{bhs}. Hence Theorem \ref{Diracequiv} ends the proof.
\end{proof}

We can now return to Example \ref{QISOZ4}. Note that the orthogonal filtration of $\bc^4$ described there corresponds to the filtration of the group $C^*$-algebra $C^*(\bz_4)$ given by the sets $\{\lambda_0\}, \{\lambda_1, \lambda_3\}$ and $\{\lambda_2\}$. Hence we obtain the following result, extending Proposition 7.3 of \cite{bsk2} and Theorem 3.2 of \cite{bhs}.

\begin{prop}
The quantum isometry group $\QISO(\widehat{\bz_4})$, for $\bz_4$ equipped with the word-length function associated to the generating set $\{1,-1\}$ is isomorphic to the free wreath product $\bz_2 \wr_* \bz_2$.
\end{prop}

The ideas from Example \ref{QISOZ4} lead to a possible technique of computing of quantum isometry groups for duals of finite abelian groups. Below we show an application of it for $\Gamma=\bz_2^{\oplus k}$ ($k\in \bn$). Recall that the quantum group $O_k^{-1}$ introduced in  \cite{bbc} was shown in that paper to be the quantum symmetry group of the $k$-dimensional cube.

\begin{tw}
Let $k\in \bn$. The quantum isometry group $\QISO(\widehat{\bz_2^{\oplus k}})$ is equal to $O^{-1}_k$.
\end{tw}

\begin{proof}
We begin by describing a general approach: let $\Gamma$ be a finite abelian group equipped with a generating set $S$ and the associated word length function $l$. Let $M=\tu{card}\, \Gamma$. Denote the elements of the dual group $\hat{\Gamma}$ by $\chi, \chi'$, etc.\ and view them as characters on $\Gamma$. Let $F:l^2(\Gamma) \to l^2(\hat{\Gamma})$ denote the Fourier transform, so that $F(\delta_{\gamma}) = \frac{1}{M} \sum_{\chi \in \hat{\Gamma}} \chi(\gamma) \delta_{\chi}$.
When we want to compute the quantum isometry group $\QISO(\hat{\Gamma})$, we ask about a universal unitary matrix $V=(V_{\gamma, \gamma'})_{\gamma,\gamma'\in \Gamma}$  with entries in some $C^*$-algebra $\blg$ such that the
\[ \alpha(\lambda_{\gamma'})=\sum_{\gamma \in \Gamma} \lambda_{\gamma} \ot V_{\gamma, \gamma'}, \;\;\; \gamma \in \Gamma\]
yields a unital $^*$-homomorphic map $\alpha:C^*(\Gamma) \to C^*(\Gamma) \ot \blg $ which additionally preserves the subspaces $H_m=\tu{Lin} \{\lambda_{\gamma}:l(\gamma)=m\}$. The latter condition is equivalent to the fact that the matrix $V$ commutes with the projections $P_m$ ($m=1,2,\ldots$), where $P_m\in B(l^2(\Gamma))$ is given by
\[ (P_m)_{\gamma, \gamma'} = \left\{\begin{array}{ccc}
1&\, \tu{if}\; \gamma=\gamma', l(\gamma)=m \\ 0 &\tu{otherwise}
\end{array}\right.\]
Now using the Fourier transform we see that this is the same as asking for a magic unitary $(V_{\chi,\chi'})_{\chi, \chi'\in \hat{\Gamma}}$ which commutes with projections $Q_m$, where $Q_m=F^{-1} P_m F$. It is easy to check that
\[ (Q_m)_{\chi, \chi'} = \sum_{\gamma\in \Gamma, l(\gamma)=m}  \overline{\chi}\chi'(\gamma), \;\;\; \chi, \chi' \in \hat{\Gamma}.\]
Due to Proposition 2.4 in \cite{banmet} instead of working with these commutation relations it suffices to understand the commutation with individual `colour components' of the matrices $Q_m$.

Let us now specialise to the case $\Gamma=(\bz^2)^{\oplus k}$, with the generating set $(1,0,\ldots), (0,1,\ldots)$, etc.. Identify $\hat{\Gamma}$ with $\Gamma$, denoting the characters on $\Gamma$ by sequences ${\bf s}\in \{0,1\}^k$ acting on $\mathbf{r}\in  \{0,1\}^k$ by
\[ {\bf s} ({\bf r}) = \prod_{i=1}^k (-1)^{s_i r_i}\]
Note that the length on $\Gamma$ takes in this case a simple form: for $s\in \Gamma$ we have $l({\bf s}) = \tu{card}\,\{i:s_i=1\}$. The matrices $Q_m$ introduced above take in this case the form:
\[   (Q_m)_{{\bf s}, {\bf t}} = \sum_{{\bf r}\in \Gamma:\,l({\bf r})=m}\; \prod_{i=1}^k(-1)^{(s_i-t_i)r_i}, \]
so that for example
\[ (Q_1)_{\bf s, \bf t} = \sum_{i=1}^k (-1)^{s_i-t_i}= k - 2 l({\bf s} - {\bf t}),\]
\[(Q_2)_{{\bf s}, {\bf t}} = \sum_{i,i'=1, \,i\neq i'}^k (-1)^{(s_i -t_i) + (s_i-t_i')} = \binom{k-l({\bf s} - {\bf t})}{2} + \binom{l({\bf s - \bf t})}{2} - (k-l({\bf s} -{\bf t}))l({\bf s} -  { \bf t}) ,\]
etc.. It is thus easy to see that `colour components' of the matrices $Q_m$ carry precisely the information about the length of the element $\bf s - \bf t$; hence, due to Proposition 2.4 of \cite{banmet}, magic unitaries commute with matrices $Q_m$ if and only they commute with the matrices $R_m$ given by the formula
\[ (R_m)_{{\bf s}, {\bf t}} = \left\{\begin{array}{ccc}
1&\, \tu{if}\;\; l({\bf s - \bf t})=m \\ 0 &\tu{otherwise}
\end{array}\right.\]
The latter condition is equivalent to the condition imposed on magic unitaries acting on the Cayley graph of $\hat{\Gamma}$, i.e.\ the $k$-dimensional cube. As in \cite{bbc} it was shown that the quantum symmetry group of the cube is $O_k^{-1}$, this ends the proof.
\end{proof}

A calculation similar to that in the proof above shows that in fact if $\Gamma=\bz_{r}^{\oplus k}$ with $r\in\{2,3,4\}$ is equipped with a standard generating set, then the quantum isometry group of $\hat{\Gamma}$ coincides with the quantum symmetry group of the Cayley graph of $\Gamma$. We believe that the same is true for arbitrary $r$.

\subsection{Free product groups, the block length and the shape}

Suppose now that $n\in \bn$, $\Gamma_1,\ldots, \Gamma_n$ are finitely generated groups and $\Gamma=\Gamma_1 \star \ldots \star \Gamma_n$. The choice of a finite generating set $S_i$ in each $\Gamma_i$ determines a natural finite set $S$ of generators of $\Gamma$, $S=\bigcup_{i=1}^n S_i$. In particular we have a natural word-length function on $\Gamma$. In this situation however, $\Gamma$ admits another natural length function, a so called \emph{block length} $b:\Gamma \to \bn_0$. The block length is defined in a following way: for any element $\gamma\in \Gamma$ we write it as a reduced word in elements in each of the groups $\Gamma_i$, and declare the length of this word to be the block length of $\gamma$; thus
\[ b(\gamma)=k\]
if
\[ \gamma= \gamma_{i_1}\cdots \gamma_{i_k}, \;\;\; i_j\in \{1,\ldots,n\}, \;\; i_j\neq i_{j+1},\;\; \gamma_{i_j} \in \Gamma_{i_j}\setminus\{e\}.\]
The idea is that each element  $\gamma_{i_j} \in \Gamma_{i_j}$ in the decomposition above corresponds to a block in $\gamma$. So for example if we are dealing with a free product of $n$-copies of $\bz$ with the corresponding generating set $S=\{\gamma_1,\ldots,\gamma_n\}$ then
\[ b(\gamma_1^{k_1} \gamma_2^{k_2} \gamma_1^{k_3})= 3,  \textup{  if only  } k_1,k_2,k_3 \neq 0.\]
For $\Gamma$ as above consider the filtration of $\Gamma$ given by the sets $F_{l,m}=\{\gamma\in \Gamma:l(\gamma)=l, b(\gamma)=m\}$ ($l,m\in \bn_0, l\leq m$). It is clear that each $F_{l,m}$ is finite and closed under taking inverses. Write $\mathcal{F}_{b}:=\{F_{l,m}: l, m\in \bn_0, l\leq m\}$.
If the action of a compact quantum group on $C^*_r(\Gamma)$  preserves the filtration $\mathcal{F}_{b}$ (in the sense of Definition \ref{presergendef}) we simply say that it \emph{preserves the word length and the block length}.

\begin{tw}
Assume that $n\in \bn$, $\Gamma_1,\ldots, \Gamma_n$ are finitely generated groups with finite generating sets denoted respectively $S_1,\ldots,S_n$ and let $\Gamma=\Gamma_1 \star \ldots \star \Gamma_n$. Then there exists a universal
compact quantum group acting on $C^*_r(\Gamma)$ in a block length and word length preserving way. We denote it $\QISO_b(\widehat{\Gamma})$; it is a quantum subgroup of $\QISO(\widehat{\Gamma})$.
\end{tw}

\begin{proof}
This follows from Theorem \ref{qsymV}, Proposition \ref{equivfiltpart} and Corollary \ref{generatedfiltrations}.
\end{proof}

Although the filtration is double-indexed, we can also use the lexicographic order to order it into a sequence, and as we have $l(\gamma\gamma')\leq l(\gamma)+l(\gamma')$, $b(\gamma\gamma')\leq b(\gamma)+b(\gamma')$ for all $\gamma, \gamma'$, we see that such ordering leads to an orthogonal filtration of $C^*_r(\Gamma)$ of spectral triple type.

Note that in the situation described above, we could also consider directly the partition of $\Gamma$ given only by the block length. In general however this fails to be a partition of $\Gamma$ into finite sets. This problem disappears if each of the $\Gamma_i$ is finite.

\begin{deft}
Let $n\in \bn$, $\Gamma_1,\ldots, \Gamma_n$ be finite groups and let $\Gamma=\Gamma_1 \star \ldots \star \Gamma_n$. An action of a compact quantum group $\QG$ on $C^*_r(\widehat{\Gamma})$ is said to preserve the block length if it preserves the filtration $\mathcal{F}_b=\{F_m:m\in \bn_0\}$, where $F_m=\{\gamma\in \Gamma:b(\gamma)=m\}$ ($m \in \bn_0$).
\end{deft}

\begin{tw} \label{QISObExistence}
Assume that $n\in \bn$, $\Gamma_1,\ldots, \Gamma_n$ are finite groups and let $\Gamma=\Gamma_1 \star \ldots \star \Gamma_n$. Then there exists a universal
compact quantum group acting on $C^*_r(\Gamma)$ in a block length preserving way. We denote it $\QIb(\widehat{\Gamma})$; it contains $\QISO_b(\widehat{\Gamma})$ as a quantum subgroup.
\end{tw}

\begin{proof}
Once again the result is a corollary of Theorem \ref{qsymV}, Proposition \ref{equivfiltpart} and Corollary \ref{generatedfiltrations}.
\end{proof}

Observe also that in the context described above the block length can be viewed as the usual word length with respect to the generating set $\bigcup_{i=1}^n \Gamma_i$ (which is of course in general non-minimal).

If $n=1$, then the condition of the block length preservation is trivial (each element of $\Gamma=\Gamma_1$ different from a neutral element has block length $1$) and the quantum symmetry group $\QIb(\widehat{\Gamma})$ is simply the quantum automorphism group in the sense of \cite{wa2} of the finite-dimensional algebra $C^*_r(\Gamma)$ equipped with its group-algebraic trace.


We saw above that for elements of $\Gamma=\Gamma_1 \star \ldots \star \Gamma_n$ (where $\Gamma_i$ are finitely generated groups, not necessarily finite) we can consider the partition of $\Gamma$ into sets of words of a given block length. This partition can be further subdivided, if apart from the number of blocks appearing in a word $\gamma\in \Gamma$ we also consider their position. Fix a finite generating set $S_i$ in each $\Gamma_i$ ($i=1,\ldots,n$). If $\gamma\in \Gamma$, $b(\gamma)=k$, ($k\neq 0$) we view $\gamma$ as a word in generators in $\bigcup_{i=1}^n S_i$ and say  that the \emph{shape} of $\gamma$, to be denoted by $s(\gamma)$, is an increasing sequence of length $k$ describing the positions where the blocks of $\gamma$ begin. Thus if $b(\gamma)=1$, we have $s(\gamma)=(l(\gamma))$; if we are once again dealing with a free product of $n$-copies of $\bz$ with the corresponding generating set $S=\{\gamma_1,\ldots,\gamma_n\}$, then for example
\[ s(\gamma_1^{k_1} \gamma_2^{k_2} \gamma_1^{k_3})= (k_1,k_1+k_2, k_1 + k_2 + k_3),  \textup{  if only  } k_1,k_2,k_3 \neq 0.\]
Note that if we in addition declare $s(e)=\emptyset$ then the shape can be viewed as a function from $\Gamma$ into the set $\Seq$, consisting of all (possibly empty) increasing finite sequences of integers. Note that this function in particular carries the information about the word length and the block length of the group elements.

Thus we can consider the family $\{F_{\mathbf{s}}:\mathbf{s}\in \Seq\}$ given by
\[ F_{\mathbf{s}}=\{\gamma \in \Gamma: s(\gamma) =\mathbf{s}\}, \;\;\; \mathbf{s}\in \Seq,\]
as a symmetric partition of $\Gamma$ into finite sets (identifying $\emptyset\in \Seq$ as the distinguished element $0$ of the index set).
Thus we can once again define what it means for an action of a compact quantum group on $C^*_r(\Gamma)$ to preserve the shape and deduce the existence of the corresponding quantum symmetry group. We leave the details to the reader. Note that  the set $\Seq$ is countable, and can be ordered  in such a way that for each $\gamma,\gamma' \in \Gamma$ if $s(\gamma)\leq s(\gamma')$, then $b(\gamma)\leq b(\gamma')$ and $l(\gamma)\leq l(\gamma')$. Hence it  can be checked that the resulting orthogonal filtration of $C^*_r(\Gamma)$ is in fact of spectral triple type.

\subsection{The universal quantum symmetry group acting on a compact quantum group $\QG$ and preserving a partition of $\Irr(\QG)$}

The new concept of a quantum symmetry group acting on a $C^*$-algebra equipped with an orthogonal filtration introduced in Section 2 allows us to extend the definition of a quantum isometry group acting on $\hat{\Gamma}$, a dual of a discrete group, in a length preserving way, to the case where $\hat{\Gamma}$ is replaced by any compact quantum group.

As before, the key observation is provided by a general lemma. Recall that for a compact quantum group $\QG$,
$\textup{Irr}(\QG)$ denote the set of all equivalence classes of irreducible unitary representations of $\QG$. We will denote the trivial representation by $0$. The following result generalises Lemma \ref{filtGamma}.

\begin{lem}
Let $\QG$ be a compact quantum group with the Haar state $h\in C(\QG)^*$ and let the family $\mathcal{F} = \{F_i:i\in \Ind\}$ be a partition of $\Irr(\QG)$ into finite sets, with $F_0=\{0\}$. Define for each $i\in \Ind$ the set $V^{\mathcal{F}}_i$ to be the span of the coefficients of all irreducible unitary representations of $\QG$ belonging to one of the equivalence classes in $F_i$.  Then the pair $(h, (V^{\mathcal{F}}_i)_{i\in \Ind})$ defines an orthogonal filtration of $C_r(\QG)$
\end{lem}
\begin{proof}
This is a consequence of the Peter-Weyl orthogonality relations from \cite{wo1}.
\end{proof}

Hence as for the duals of discrete groups, given a partition of $\Irr(\QG)$ we can define the universal quantum symmetry group of $\QG$ (viewed as a quantum space, not a quantum group!) for the actions preserving the associated partition. In particular, using the concept of length on a discrete quantum group introduced and developed in \cite{Vergnioux}, we can define and study the quantum symmetry group of a given compact quantum group $\QG$ with respect to a fixed length on $\hat{\QG}$.

\section{The quantum isometry group $K^+_n$ and its interpretation as a quantum symmetry group}

In this section we describe a compact quantum group $K_n^+$,  first discovered in \cite{bsk2} via its category of representations (given by a certain class of non-crossing partitions) and show it is the quantum symmetry group $\QIb(\widehat{\bF_n})$.

Whenever we consider matrices of the size $2n$ we will denote the relevant indices by pairs $i\kappa$,
where $i\in\{-1,1\}$ and $\kappa\in\{1,\ldots,n\}$.
Moreover we will use the notation $\overline{i}=i^{-1}$, so that $\overline{\overline{i}}=i$. This
facilitates the description of the fact that coordinates  come naturally in pairs. To simplify the notation we
will also write $\Jnd_n=\{i\kappa:i\in \{-1,1\},\alpha\in\{1, \ldots, n\}\}$.  The canonical basis in $\bc^{2n}$ will be denoted by $(e_z)_{z \in \Jnd_{n}}$.

Recall that if $U=(U_{ij})_{i,j=1}^k \in M_k(\blg)\approx B(\bc^k)\ot \blg$ is a unitary matrix with entries in some $C^*$-algebra $\blg$, the intertwiner space
$\tu{Hom} (U^{\ot l}; U^{\ot m})$ is defined as a collection of all operators $T\in B((\bc^k)^{\ot l};(\bc^k)^{\ot m})$ such that
\[  (T\ot \id_{\blg})(U^{\ot l}) = U^{\ot m}(T\ot \id_{\blg}),\]
where in the last formula above the notation $U^{\ot k}$ refers to the usual $k$-fold product of representations of compact quantum groups, see for example \cite{wo1}, so that $U^{\ot k} \in B((\bc^k)^{\ot l})\ot \blg$.

We also need to recall the following definition introduced in \cite{bsk1}.

\begin{deft}\label{definitHn}
Let $n \in \bn$. The compact quantum group $H^+(n,0)$ is the compact quantum group whose fundamental unitary representation $(U_{i\kappa,j\beta})_{i\kappa, j\beta \in \Jnd_n}$ is the universal unitary matrix such that for each $i\kappa, j\beta \in \Jnd_n$
\begin{rlist}
\item $U_{i\kappa,j\beta}$ is a partial isometry;
\item $U_{i\kappa,j\beta} = U_{\bar{i}\kappa,\bar{j}\beta}^*$.
\end{rlist}
\end{deft}

Recall also that in \cite{bsk1} we showed that $H^+(n,0)=\QISO(\freedual)$. The key result for this section is given by the following lemma.

\begin{lem} \label{mainlemma}
Let $\QG$ be a compact quantum group which has the fundamental unitary representation $U=(U_{x,z})_{x,z\in \Jnd_n}$ satisfying the defining conditions of the fundamental representation of $H^+(n,0)=\QISO(\freedual)$ (so that $\QG$ is a quantum subgroup of $H^+(n,0)$). Consider the action of $\QG$ on $C^*(\bF_n)$ given by the usual formula:
\begin{equation} \alpha (\gamma_{\kappa}) = \sum_{i\sigma\in \Jnd_n} \gamma_{\sigma}^i \ot  U_{i\sigma, 1\kappa},\label{Fnaction}\end{equation}
where $\gamma_1,\ldots,\gamma_n$ denote the canonical generators of $\bF_n$.
Then the following conditions are equivalent:
\begin{rlist}
\item the operator $T:\bc^{2n} \otimes \bc^{2n} \to \bc^{2n} \ot \bc^{2n}$ given by
\[ T (e_z \ot e_x) = \delta_{z\bar{x}} e_{\bar{z}} \otimes e_z, \;\;\; x,z \in \Jnd_n\]
is an intertwiner in $\textup{Hom} (U^{\ot 2}; U^{\ot 2})$;
\item each $U_{x,z}$ is a normal operator (so in particular a `partial unitary');
\item the action $\alpha$ preserves the block length of elements of $\freedual$ which have both the block and the word length equal $2$;
\item the action $\alpha$ preserves the block length of elements of $\freedual$;
\item the action $\alpha$ preserves the shape of elements of $\freedual$.
\end{rlist}
\end{lem}

\begin{proof}
 Assume that $U=(U_{x,z})_{x,z\in \Jnd_n}$ is a unitary matrix of partial isometries satisfying the conditions in Proposition 2.3 of \cite{bsk2}.

(i)$\Longleftrightarrow$ (ii) This equivalence was already stated in \cite{bsk2}. The intertwining condition $(T\ot \id)(U^{\ot 2}) = (U^{\ot 2}) (T\ot \id)$ (where $U$ is viewed as a matrix in $B(\bc^{2n} \ot C(H^+(n,0))$) is equivalent to the following equality holding for all $x,y,z,t \in \Jnd_n$
\[ \delta_{y, \bar{z}} U_{x,\bar{y}} U_{t,y} = \delta_{x, \bar{t}} U_{\bar{x},y} U_{\bar{t},z}.\]
Considering the adjoints we see that the above can be rewritten as
\[ \delta_{y,z} U_{x,y}^* U_{t,y} = \delta_{x,t} U_{x,y} U_{t,z}^*, \;\;\;\; x,y,z,t \in \Jnd_n\]
and further simply as (because the entries of $U$ satisfy the condition (2) in Proposition 2.3 of \cite{bsk2})
\[U_{x,y} U_{x,y}^* = U_{x,y}^* U_{x,y}\;\;\;\; x,y\in \Jnd_n.\]

(ii)$\Longleftrightarrow$ (iii)  From the formula for the action (and the fact that we are dealing with  the free product of groups it follows that for $\beta, \sigma, \kappa \in \{1,\ldots,n\}$, $\sigma \neq \kappa$, and $i,j,l \in \{-1,1\}$
\[ f_{\gamma_{\sigma}^j \gamma_{\kappa}^l} (\alpha (\gamma_{\beta}^{2i})) = U_{j\sigma,i\beta} U_{l \kappa, i\beta}.\]
Hence the fact that $\alpha$ preserves the block length of elements of $\freedual$ which have both the block and the word length equal $2$ is equivalent to the condition
\begin{equation}
U_{j\sigma,x} U_{l \kappa, x}=0, \; \; x\in \Jnd_n, j,l \in \{-1,1\}, \sigma, \kappa \in \{1,\ldots,n\}, \sigma \neq \kappa. \label{block2}
\end{equation}
If (ii) holds then for each $x,y,z \in \Jnd_n$, $y\neq z$ we have
\[ U_{y,x} U_{z,x} = U_{y,x} U_{z,x} U_{z,x}^* U_{z,x} = U_{y,x} U_{z,x}^* U_{z,x} U_{z,x} =0,\]
so that  \eqref{block2} holds.
Assume then that \eqref{block2} holds and for each $x,y \in \Jnd_n$ consider the projections $P_{y,x}= U_{y,x} U_{y,x}^*$ and $Q_{y,x} = U_{y,x}^* U_{y,x}$. Recall that
\[Q_{y,x} P_{\bar{y}, x} = Q_{y,x} P_{y, \bar{x}}=0,\]
so that in fact \eqref{block2} implies that
\[ Q_{y,x} P_{z, x} =0, \;\;\;\; x,y,z \in \Jnd_n, y \neq z.\]
Hence
\[ P_{x,y} Q_{x,y} = P_{x,y} (1-\sum_{z\in \Jnd_n} Q_{x,z})= P_{x,y}\]
and similarly \[P_{x,y} Q_{x,y} = (1-\sum_{z\in \Jnd_n} P_{x,z}) Q_{x,y} = Q_{x,y},\]
which proves (ii).

(iii)$\Longrightarrow$(v) Suppose that (ii) holds. Let $v,w\in \bF_n$ be two words which have different shapes. We want to show that then  we must have $f_w (\alpha(v)) =0$. We may assume that $l(w)=l(v)=k$ and that both words $v,w$ are in the reduced form. As the action $\alpha$ is assumed to preserve the usual length,  we actually have
\[f_w (\alpha(v))= f_{w} (\alpha(v_1) \cdots\alpha(v_n))  = f_{w_1} (\alpha(v_1)) \cdots f_{w_n} (\alpha(v_n)).\]
But as $v$ and $w$ have different shapes, there must be some $j\in \{1,\cdots,n-1\}$ such that $b(v_j v_{j+1})=2\neq b(w_j w_{j+1})=1$.
As the words  $v'=v_j v_{j+1}$ and $w'=w_j w_{j+1}$ are again reduced, we must have
\[f_{w'}(\alpha (v')) = f_{w_j} (\alpha(v_j)) f_{w_{j+1}} (\alpha(v_{j+1})).\]
But from (iii) it follows that
\[ f_{w'} (\alpha (v')) =0\]
and the last three displayed formulas end the proof.

(v)$\Longrightarrow$(iv)$\Longrightarrow$(iii) This is obvious.
\end{proof}


\begin{deft} \label{defKn}
The algebra $A_k(n)$ is the universal $C^*$-algebra generated by elements $\{U_{x,y}, x,y \in \Jnd_{n}\}$ which satisfy the following conditions:
\begin{rlist}
\item $U_{x,y}^*=U_{\bar{x}, \bar{y}}$, $x,y \in \Jnd_n$;
\item each $U_{x,y}$ is a normal partial isometry;
\item the matrix $(U_{x,y} U_{x,y}^*)_{x,y\in \Jnd_n}$ is a magic unitary.
\end{rlist}
\end{deft}

\begin{tw} \label{DefKN}
The algebra $A_k(n)$ is  the algebra of continuous functions on a compact quantum group, denoted further $K^+_n$. The unitary $(U_{x,y})_{x,y \in \Jnd_{n}} \in M_{2n}(C(K_n^+))$ is a fundamental unitary representation of $K_n^+$. The quantum group $K^+_n$ is the universal quantum group acting on $\freedual$ and preserving both the word-length and the block-length; in other words $K^+_n=\QISOb(\freedual)$.
\end{tw}

\begin{proof}
The fact that the algebra $A_k(n)$ is the algebra of continuous functions on a compact matrix quantum group, with $(U_{x,y})_{x,y \in \Jnd_{n}}$ as a fundamental unitary representation follows from the fact that the conditions in Definition \ref{defKn} are stable under replacing $U_{x,y}$ by $\sum_{z\in \Jnd_n} U_{x,z} \ot U_{z,y}$ (corresponding to the coproduct), $U_{x,y}$ by $U^*_{y,x}$ (corresponding to the antipode) and $U_{x,y}$ by $\delta_{x,y} 1$ (corresponding to the counit). The fact that the resulting compact quantum group is  indeed equal to $\QISOb(\freedual)$ follows from the equivalence of conditions (ii) and (iv) in Lemma \ref{mainlemma}.
\end{proof}

Note that the implication (iv)$\Longrightarrow$(v) in Lemma  \ref{mainlemma} shows that the canonical action of $K^+_n$ on $C^*(\bF_n)$ preserves also the shape.

As mentioned in the beginning of this section, the quantum group $K_n^+$ was in fact discovered in \cite{bsk2} via its category of representations, described in terms of non-crossing partitions. The following result is the first half of Theorem 6.5 of \cite{bsk2}; $D_{\infty}$ denotes a certain category of non-crossing partitions with legs coloured black or white, with the so-called signed numbers of white and black legs equal. To each $\pi \in D_{\infty}(k,l)$ one can associate a unique operator $T_{\pi}\in B((\bc^{2n})^{\ot k}; (\bc^{2n})^{\ot l})$ (the details can be found in Section 3 of \cite{bsk2}).

\begin{tw}
Let $K^+_n$ be the quantum group introduced in Theorem \ref{DefKN}. Then for any $k,l\in\mathbb N_0$ we have
$\textup{Hom}(U^{\otimes k},U^{\otimes l})=\textup{span}(T_\pi|\pi\in D_\infty(k,l))$.
\end{tw}


\begin{rem}
The above theorem, together with the results on representation theory of the two-parameter families of compact quantum groups obtained in \cite{bsk1} and the considerations in \cite{bcz} suggests that there exists a natural generalization of the class of \emph{easy quantum groups} introduced in \cite{bsp}. Recall that informally speaking easy quantum groups are those whose spaces of intertwiners are spanned by certain partitions; the new, broader class, should also admit representation theories described by coloured partitions, as it is for example for $K^+_n$, and connect the theory developed in \cite{bsk1} to another possible (un)twisting related to the Drinfeld-Jimbo parameter appearing in the implementation $\pi \to T_{\pi}$.
\end{rem}

\vspace*{0.5cm}

\noindent \emph{Acknowledgments.}
The work of T.B.\  was supported by the ANR grants Galoisint and Granma. A.S.\ was partially supported by National Science Center (NCN) grant \\no.~2011/01/B/ST1/05011.

\end{document}